\documentclass[10pt]{amsart}

\usepackage[utf8]{inputenc}
\usepackage[left=4 cm, right= 4 cm]{geometry}
\usepackage{lmodern}
\usepackage[english,activeacute]{babel}
\usepackage{empheq}
\usepackage{xcolor}
\usepackage{faktor}
\usepackage{amsmath}
\usepackage{amsrefs}
\usepackage{yhmath}
\usepackage{amsthm}
\usepackage{amssymb}
\usepackage{calligra,mathrsfs}
\usepackage{enumitem}
\usepackage{extpfeil}
\usepackage{graphics}
\usepackage[all]{xy}
\usepackage{hyperref}
\usepackage{bookmark}
\usepackage{bbm}

\newtheorem{theorem}{Theorem}[section]
\newtheorem{corollary}[theorem]{Corollary}
\newtheorem{lemma}[theorem]{Lemma}
\newtheorem{prop}[theorem]{Proposition}

\newtheorem{ejemplo}[theorem]{Example}
\newtheorem{comentario}[theorem]{Remark}

\newtheorem{pregunta}[theorem]{Question}

\theoremstyle{definition}
\newtheorem{definicion}[theorem]{Definition}

\numberwithin{equation}{section}

\newcommand*{\myproofname}{Proof of the Lemma:}
\newenvironment{myproof}[1][\myproofname]{\begin{proof}[#1]}{\end{proof}}

\newcommand*{\myproofteo}{Proof of the Theorem:}
\newenvironment{myprooft}[1][\myproofteo]{\begin{proof}[#1]}{\end{proof}}

\newcommand*{\myproofprop}{Proof of the Proposition:}

\newcommand{\R}{\mathbb{R}}

\newcommand{\Q}{\mathbb{Q}}

\newcommand{\Z}{\mathbb{Z}}

\newcommand{\NS}{\mathrm{NS}}
\newcommand{\Pic}{\mathrm{Pic}}

\newcommand*{\shtor}{\mathscr{T}\kern -.5pt or}
\newcommand*{\shom}{\mathscr{H}\kern -.5pt om}

\title[Jets-thresholds, Seshadri constants and Wahl maps on abelian varieties]{Jets-separation thresholds, Seshadri constants and higher Gauss-Wahl maps on abelian varieties}
\author{Nelson Alvarado}
\address{N. Alvarado \\Dipartimento di Matematica, Università degli studi di Roma Tor Vergata , Via della ricerca scientifica, 00133 Roma, Italy}
\email{alvarado@mat.uniroma2.it}

\begin{document}

\maketitle




\begin{abstract}
Given a closed subscheme $Z$ of a polarized abelian variety $(A,\ell)$ we define its \emph{vanishing threshold with respect to $\ell$} and relate it to the Seshadri constant of the ideal defining $Z.$ As a particular case, we introduce the notion of \emph{jets-separation threshold}, which naturally arises as the \emph{vanishing threshold} of the $p$-infinitesimal neighborhood of a point. Afterwards, by means of Fourier-Mukai methods we relate the jets-separation thresholds with the surjectivity of certain higher Gauss-Wahl maps. As a consequence we obtain a criterion for the surjectivity of those maps in terms of the Seshadri constant of the polarization $\ell.$
\end{abstract}

\vspace{0.3cm}



\maketitle

\section{Introduction}
In this article we introduce the \emph{vanishing} threshold of a closed subscheme of a polarized abelian variety defined over an algebraically closed field. In the case of a closed point (with its reduced structure) this number is no other than the basepoint-freeness threshold introduced by Jiang-Pareschi in \cite{cohrank}, which has also been studied by Caucci \cite{caucci}, Ito \cite{itobir} and Jiang \cite{jiang}, using methods from birational geometry, and by Rojas \cite{Rojas}, by means of stability conditions in the two-dimensional case. When we consider the $p$-infinitesimal neighborhood of a closed point (i.e the scheme defined by the $(p+1)$-power of the ideal of the point) a natural notion of $p$-jets separation threshold arises. 

In the first part of this work, we introduce the aforementioned thresholds and show some examples.  Afterwards, we prove that if $Z$ is a closed subscheme and $Z^{(p)}$ is its $p$-infinitesimal neighborhood then the sequence, suitably normalized, of vanishing thresholds of $Z^{(p)}$ with respect to a polarization $\ell$ converges to the Seshadri constant with respect to $\ell$ of the ideal defining $Z.$ As a particular case, when $Z$ is a point (with its reduced structure) we get that the sequence of $p$-jets separation thresholds (suitably normalized) converges to the classical Seshadri constant of $\ell.$ This is no surprise since, by a theorem of Demailly (\cite[Theorem 6.4]{Dem}), the Seshadri constant of a line bundle at a point reflects the asymptotic behaviour of jet separation. However, unlike the invariants $s(kL,x)$ considered by Demailly (see \cite[Proposition 5.1.6]{PositivityI} for the definition of these numbers), which are integers and depend on $x$, the $p$-jets separation thresholds are real numbers and do not depend on an specific closed point $x.$ In this context, our asymptotic description can be seen as a finer version, valid for abelian varieties, of Demailly's result.

Subsequently, we show that $p$-jets separation thresholds are intimately related, via the Fourier-Mukai transform, to certain higher Gauss-Wahl maps. In fact, we prove that they carry esentially the same information of certain (suitably defined) \emph{$p$-Gauss-Wahl surjectivity threshold}. This paralells the fact, shown by Jiang-Pareschi in \cite[Proposition 8.1]{cohrank}, that basepoint freeness threshold carries the same information as the surjectivity threshold for multiplication maps of global sections (and therefore, is intimately related to projective normality). As a byproduct of the aforementioned relation we obtain a description of the Seshadri constant in terms of the asymptotic behaviour of the surjectivity of certain Gauss-Wahl maps. Finally, by means of Nadel's vanishing theorem, we bound the $p$-jets separation thresholds in terms of the Seshadri constant, when we work in characteristic zero. As a consequence, we establish precise lower bounds for the Seshadri constant guaranteeing the surjectivity of higher Gauss-Wahl maps, asymptotically improving the results present in the literature. 

Turning to a more detailed presentation, we now proceed to define the afore mentioned thresholds. First, we recall that a coherent sheaf $\mathcal{F}$ on an abelian variety $A$ is said to satisfy the index theorem with index zero (IT(0) for short) if 
\begin{equation*}
H^{i}(A,\mathcal{F}\otimes P_{\alpha}) = 0 \hspace{0.3cm} \text{for every $i>0$ and for every $\alpha\in\mathrm{Pic}^{0}(A),$}
\end{equation*} 
where $P_{\alpha}$ is the line bundle on $A$ parametrized by $\alpha$ via a normalized Poincar\'e bundle. 
Now, let $\ell$ be a polarization on $A$ (that is, an ample class in $\NS(A)$).  As shown by Jiang-Pareschi, the IT(0) condition can be extended for $\Q$-twisted sheaves $\mathcal{F}\left<t\ell\right>$ (we recall this notion in Section 2.1). In this context, we define the \emph{vanishing threshold} of $\mathcal{F}$ with respect to $\ell$ as the real number
\begin{equation}
\label{def-vanishing-threshold}
\nu_{\ell}(\mathcal{F}) = \mathrm{inf}\left\{t\in\Q_{\geq 0} : \mathcal{F}\left<t\ell\right>\hspace{0.2cm}\text{satisfies IT(0)}\right\}.
\end{equation}
In this context, we define the vanishing threshold of a closed subscheme $Z$ of $A$ with respect to $\ell$ to be the vanishing threshold of its ideal. In symbols:
$$\beta(Z,\ell) : = \nu_{\ell}(I_{Z}).$$
As an example, fix a closed point $x\in A$ with ideal sheaf $I_{x}$ and consider a positive integer $p.$ Let $L$ be a line bundle representing $\ell.$ As $L$ is ample, all the line bundles $L\otimes P_{\alpha}$ are isomorphic to $t_{y}^{\ast}L$ for some closed point $y\in A,$ where $t_{y}$ is the translation map, and hence the vanishing 
\begin{equation}
\label{vanishing}
H^{1}(A,I_{x}^{p+1}\otimes L\otimes P_{\alpha}) = 0\hspace{0.3cm}\text{for every $\alpha\in\Pic^{0}(A)$}
\end{equation}
means that for every $y\in A$ the line bundle $t_{y}^{\ast}L$ separates $p$-jets at $x$ and thus, that $L$ separates $p$-jets at every point of $A.$ As the higher cohomology groups of $I_{x}^{p+1}\otimes L\otimes P_{\alpha}$ authomatically vanish, the condition \eqref{vanishing} turns to be equivalent to IT(0) and therefore, it is natural to define the $p$-jets separation threshold of $\ell$ using the vanishing threshold with respect to $\ell$ of the $p$-infinitesimal neighborhood of $x$, where $x$ is a (any) closed point of $A,$  that is:
$$\beta_{p}(\ell):= \beta(x^{(p)},\ell) = \nu_{\ell}(I_{x}^{p+1}).$$
Our first result relates the vanishing threshold of infinitesimal neighborhoods of a closed subscheme $Z$ with the Seshadri constant $\varepsilon(I_{Z},L)$ of the ideal defining $Z$ (see section 4 for its definition). More precisely, we prove the following:
\begin{theorem}
\label{A}
Let $A$ be a $g$-dimensional abelian variety and $Z\subset A$ a closed subscheme with ideal $I_{Z}.$ For $p\in\Z_{\geq 0}$ write $Z^{(p)}$ for the closed subscheme of $A$ defined by the ideal $I_{Z}^{p+1}.$ Let $L$ be an ample line bundle on $A$ with class $\ell\in\NS(A),$ then:
\begin{enumerate}
\item\hspace{0.1cm}  The following inequalities hold:
$$\sup \frac{p+1}{\beta(Z^{(p)},\ell)}\geq \varepsilon(I_{Z},L)\geq \lim\sup\frac{p+1}{\beta(Z^{(p)},\ell)}.$$
\item\hspace{0.1cm}  If $\dim Z\leq 1$ then 
\begin{enumerate}
\item\hspace{0.1cm} For $p,r\in\Z_{\geq 0}$ with $p<r$ we have
$$0\leq \beta(Z^{(r)},\ell)\leq\beta(Z^{(r-p-1)},\ell)+\beta(Z^{(p)},\ell)$$
\item\hspace{0.1cm} The sequence $\{(p+1)^{-1}\beta(Z^{(p)},\ell)\}_{p\in\Z_{\geq 0}}$ converges and 
$$\varepsilon(I_{Z},L) = \lim_{p\to\infty}\frac{p+1}{\beta(Z^{(p)},\ell)}=\sup_{p}\frac{p+1}{\beta(Z^{(p)},\ell)}.$$
\end{enumerate}
\item\hspace{0.1cm} If the base field has characteristic zero, the following inequalities of jets separation thresholds hold:
$$ \frac{p+1}{\varepsilon(\ell)} \leq \hspace{0.1cm} \beta_{p}(\ell)\hspace{0.1cm}\leq \hspace{0.1cm}\frac{g+p}{\varepsilon(\ell)},$$
where $\varepsilon(\ell)$ is the Seshadri constant of $\ell.$
\end{enumerate}
\end{theorem} 

Our second result concerns higher Gauss-Wahl maps. They were introduced by Wahl in \cite{Wahl} and they are a hierarchy of linear maps $\{\gamma_{L,M}^{p}\}_{p\in\Z_{\geq 0}}$ associated to a pair of line bundles $L,M$, where $\gamma_{L,M}^{0}$ is no other than the multiplication map of global sections. For the moment, we give an informal introduction to these maps (in Section 5 we recall their precise definition). In brief, given two line bundles $L,M$ on a projective variety $A,$ the associated $p$-th Gauss-Wahl map is a linear map 
$$\gamma_{L,M}^{p}: H^{0}\left(R_{L}^{(p-1)}\otimes M\right)\rightarrow H^{0}(\operatorname{Sym}^{p}\Omega_{A}\otimes L\otimes M),$$
where $R_{L}^{(p-1)}$ is a sheaf that in the case that $L$ is very ample is no other than the $p-1$-order conormal bundle associated to the embedding of $A$ into $\mathbb{P}H^{0}(L)$ (for instance, $R_{L}^{(1)}$ can be identified with $N_{A/\mathbb{P}H^{0}(L)}^{\vee}\otimes L$).

When $L$ separates $p$-jets at every point, the surjectivity of $\gamma_{L,M}^{p}$ is detected by the vanishing of $H^{1}(R_{L}^{(p)}\otimes M)$ (we refer to section 5 for details). In this context, informally speaking, the vanishing threshold $\nu_{\ell}(R_{L}^{(p)})$ (see \eqref{def-vanishing-threshold}) of the sheaf $R_{L}^{(p)}$ with respect to $\ell$ may be seen as the minimal ``rational power'' of $L$ such that the ``rational'' Gauss-Wahl map $\gamma_{L,\nu L}^{p}$ is surjective. Extending \cite[Proposition 8.1]{cohrank}, our second result says that the numbers
$$\mu_{p}(\ell) := \nu_{\ell}\left(R_{L}^{(p)}\right)$$
are related, via the Fourier-Mukai transform (reviewed in Section 2.3), with the jets-separation thresholds as follows:
\begin{theorem}
\label{B}
Let $L$ be an ample line bundle on an abelian variety $A.$ Let $\ell\in\NS(A)$ be the class of $L$ and suppose that $L$ separates $p$-jets at every point of $A$ (thus $\beta_{p}(\ell)<1$). Then the following equality holds:
$$\mu_{p}(\ell) = \frac{\beta_{p}(\ell)}{1-\beta_{p}(\ell)}.$$
\end{theorem}

As a consequence, this gives a surprising expression of the Seshadri constant in terms of the asymptotic surjectivity of higher Gauss-Wahl maps. Concretely, we have the following:

\begin{corollary}
\label{sesh-gauss}
Let $\ell\in\NS(A)$ be a polarization on an abelian variety $A.$ Then:
$$\varepsilon(\ell) = 1 + \lim_{p\to\infty} \frac{1}{\mu_{p}((p+2)\ell)}$$
\end{corollary}

An interesting consequence of this relation is the fact that we are able to characterize special kinds of polarized abelian varieties using Gauss-Wahl maps. Concretely, here we highlight two special cases:

\textbf{1. Decomposable abelian varieties with a curve factor:} In \cite{Nakamaye} Nakamaye proved that if $(A,\ell)$ is a polarized abelian variety, then $\varepsilon(\ell)\geq 1$ and, moreover, $\varepsilon(\ell) = 1$ if and only if 
\begin{equation}
\label{tipo-producto}
(A,\ell)\simeq (E,\theta)\boxtimes(B,m),
\end{equation}
where $(E,\theta)$ is a principally polarized elliptic curve and $(B,m)$ is another polarized abelian variety. In this context, Corollary \ref{sesh-gauss} says that varieties of the form \eqref{tipo-producto} are also characterized by the following equivalent conditions:
\begin{enumerate}[label = \alph*)]
\item\hspace{0.1cm} $\beta_{p}(\ell) = p+1$ for every $p\in\Z_{\geq 0}$ 
\item\hspace{0.1cm} the sequence $\{\mu_{p}((p+2)\ell)\}_{p\in\Z_{\geq 0}}$ is unbounded
\end{enumerate}

\textbf{2. Hyperelliptic jacobians:} In \cite[Theorem 7]{Deb} it is shown that if $C$ is a curve of genus $g\geq 5$ and $C$ is not hyperelliptic then $\varepsilon(\operatorname{Jac}C,\theta)>2$  while $\varepsilon(\operatorname{Jac}C,\theta) = 2g/(g+1)<2$ in the hyperelliptic case. Moreover, in the cited reference it is conjectured that the condition $\varepsilon(\theta)<2$  should characterize jacobians of hyperelliptic curves among indecomposable principally polarized abelian varieties (i.p.p.a.v) of dimension $g\geq 5.$ 

In this context, Corollary \ref{sesh-gauss} says that jacobians of hyperelliptic curves are characterized among jacobians (and conjecturally among all i.p.p.a.v) by the fact that $\mu_{p}((p+2)\theta)>1$ for $p\gg 0,$ that is, for the failure of surjectivity of the Gauss-Wahl map $\gamma_{(p+2)\theta,(p+2)\theta}^{p}$ for $p\gg 0.$

Finally, as an application, we use Theorem \ref{B} above to study the surjectivity of Gauss-Wahl maps $\gamma_{cL,dM}^{p},$ where $c,d,p\in\Z_{\geq 0}$ and $L,M$ are ample and algebraically equivalent line bundles. Our result is the following:

\begin{theorem}
\label{resultado-sobreyectividad}
Let $L,M$ ample and algebraically equivalent ample line bundles on an abelian variety $A.$ Let $c,d$ be positive integers and write $\ell$ for the class of $L$ and $M$ in $\NS(A).$ Assume that 
$$\beta_{p}(\ell) < \frac{cd}{c+d}$$
for some $p\in\Z_{\geq 0}.$ Then the Gauss-Wahl map $\gamma_{cL,dM}^{p}$ is surjective.
\end{theorem}

For instance, the above result says that if $\beta_{p}(\ell)<1,$ that is, if $\ell$ separates $p$-jets at every point, then $\gamma_{2L,2M}^{p}$ is surjective, while $\gamma_{L,M}^{p}$ is already surjective whenever $\beta_{p}(\ell)<1/2.$ On the other hand, using Theorem \ref{A} 3) we obtain a condition ensuring the surjectivity of certain Gauss-Wahl maps in terms of the Seshadri constant, which (at least asymptotically) improves the results present in the literature (\cite[Theorem 2.2]{mult}, for example). More precisely, we get the following:

\begin{corollary}
\label{corolario-sobreyectividad-seshadri}
Let $\ell$ be a polarization on an abelian variery. Consider a positive integer $c$ such that $c\cdot\varepsilon(\ell)>g+p,$ where $\varepsilon(\ell)$ is the Seshadri constant of $\ell.$ Then 
$$\mu_{p}(c\ell) < \frac{g+p}{c\varepsilon(\ell)-(g+p)}.$$
In particular, for $L,M$ ample line bundles representing $\ell,$ the Gauss-Wahl map $\gamma_{cL,dM}^{p}$ is surjective as soon as 
$$d\hspace{0.1cm}\geq\hspace{0.1cm}\frac{c(g+p)}{c\varepsilon(L)-(g+p)}.$$
\end{corollary}

This article structures as follows: in Section 2 we survey the notions of generic vanishing for $\Q$-twisted objects and their relation with the Fourier-Mukai functor, whose definition we also recall. In Section 3 we formally introduce the notion of vanishing threshold of a closed subscheme of an abelian variety and, in particular,  the $p$-jets separation thresholds. In this context, we establish Theorem \ref{A} 2.a). In Section 4 we discuss Seshadri constants of ideal sheaves, and relate them with the vanishing threshold of the corresponding closed subscheme, proving Theorem \ref{A}.1) and 2.b). In Section 5 we discuss Gauss-Wahl maps, introduce their surjectivity thresholds and prove a slightly more general version of Theorem \ref{B} above by computing the Fourier transform of $I_{0}^{p+1}\otimes L.$  In Section 6, we prove Theorem \ref{A} 3) and show its consequences. In particular, in this section we prove Theorem \ref{resultado-sobreyectividad} and Corollary \ref{corolario-sobreyectividad-seshadri}. Finally, in Section 7 we state a couple of questions for future research: the first one regards the possibility to render our asymptotic results more effective and the second one regards the possibility to extend Theorem \ref{A} 3) to abelian varieties defined over fields of positive characteristic.

\section{Preliminaries}

In this section we recall the preliminary materials that we will need. We start by fixing some notation.  Let $A$ be an abelian variety defined over an algebraically closed field $k$. By \emph{a polarization} on $A$ we mean an ample class $\ell\in\NS(A).$ For $A$ as before, we write $\hat{A}$ for the dual of $A,$ that is, $\hat{A} = \Pic^{0}(A)$ is the abelian variety parametrizing traslation-invariant line bundles on $A.$ For a polarization $\ell$ we write $\varphi_{\ell}$ (or $\varphi_{L}$) for the isogeny $A\rightarrow\hat{A}$ which sends a point $x\in A$ to the point of $\hat{A}$ corresponding to the line bundle $t_{x}^{\ast}L\otimes L^{-1},$ where $t_{x}:A\rightarrow A$ is the traslation map given by $z\mapsto z+x$ and $L$ is a (any) line bundle representing $\ell$. For a positive integer $n$ we write $\lambda_{n}:A\rightarrow A$ for the isogeny $x\mapsto nx.$ Finally, we fix $\mathcal{P}$ a normalized Poincar\'e bundle on $A\times\hat{A}$ and for $\alpha\in\hat{A},$ write $P_{\alpha}$ for the line bundle on $A$ parametrized by $\alpha$ (equivalently, $P_{\alpha}$ is the fiber of $\mathcal{P}$ over $\alpha$). 

\subsection{$\Q$-twisted objects}

For $A$ as before, we write $D^{b}(A)$ for the bounded derived category of coherent sheaves on $A.$ We consider pairs $(\mathcal{F},t\ell),$ where $\mathcal{F}$ is an object in $D^{b}(A),$ $t$ is a rational number and $\ell\in\NS(A)$ is a polarization. On these pairs we impose the relation given by identifying  $(\mathcal{F}\otimes L^{\otimes m},t\ell)$  with the pair $(\mathcal{F},(t+m)\ell)$ for every integer $m$ and for every $L$ representing $\ell.$ 

\begin{definicion}
A $\Q$-twisted object, written $\mathcal{F}\left<t\ell\right>,$ is the class of the pair $(\mathcal{F},t\ell)$ under the relation described in the previous paragraph.
\end{definicion}

For a $\Q$-twisted object $\mathcal{F}\left<t\ell\right>$ we would like to define its cohomological support loci $V^{i}(\mathcal{F}\left<t\ell\right>).$ To start, let $L$ be a line bundle representing $\ell$ and consider $a,b\in\Z$ with $b>0$ and $t=a/b.$ Write 
$$V^{i}(\mathcal{F},L,a,b) :=\left\{\alpha\in\hat{A} : H^{i}(\lambda_{b}^{\ast}\mathcal{F}\otimes L^{ab}\otimes P_{\alpha})\neq 0\right\},$$  
where $H^{i}$ stands for the sheaf (hyper)cohomology. These subsets depend on the choices of $L,a$ and $b.$ For instance, if $N$ is another representant of $\ell$ then $V^{i}(\mathcal{F},N,a,b)$ is just a traslation of $V^{i}(\mathcal{F},L,a,b).$ On the other hand, if we change $a,b$ for $ac,bc$ with $c$ a positive integer not divisible by $\operatorname{char} k$ and we consider a symmetric $L$, then: 
\begin{align*}
V^{i}(\mathcal{F},L,ac,bc) & = \left\{\alpha\in\hat{A} : H^{i}(\lambda_{bc}^{\ast}\mathcal{F}\otimes L^{c^{2}ab}\otimes P_{\alpha})\neq 0\right\} \\
& = \left\{\alpha\in\hat{A} : H^{i}(\lambda_{c}^{\ast}(\lambda_{b}^{\ast}\mathcal{F}\otimes L^{ab}\otimes P_{\gamma}))\neq 0\hspace{0.2cm}\text{for a (any) $\gamma\in \lambda_{c}^{-1}(\alpha)$}\right\} \\
& = \left\{\alpha\in\hat{A} : \bigoplus_{\gamma\in \lambda_{c}^{-1}(\alpha)} H^{i}(\lambda_{b}^{\ast}\mathcal{F}\otimes L^{ab}\otimes P_{\gamma})\neq 0\right\}, \\
& = \lambda_{c}\left(V^{i}(\mathcal{F},L,a,b)\right).
\end{align*}
where, to get the direct sum we used \cite[p.72]{Mumford} and the projection formula. However, we can make the following:

\begin{definicion}
For a $\Q$-twisted object $\mathcal{F}\left<t\ell\right>$ its $i$-th cohomological support loci $V^{i}\left(\mathcal{F}\left<t\ell\right>\right)$ is the class of $V^{i}(\mathcal{F},L,a,b),$ where $L$ is a representant of $\ell$ and $a,b$ are integers with $a/b = t$ and $\operatorname{gcd}(a,b)=1$, under the equivalence relation generated by translations and by taking direct images of prime to the characteristic multiplication isogenies (however, see Section 2.2 below). 
\end{definicion}

In particular, although $V^{i}(\mathcal{F}\left<t\ell\right>)$ is not a set, its dimension and codimension are well defined and they are denoted by $\dim V^{i}(\mathcal{F}\left<t\ell\right>)$ and $\mathrm{codim}_{\hat{A}}V^{i}(\mathcal{F}\left<t\ell\right>),$ respectively. In \cite[Section 5]{cohrank} the cited authors use these loci in order to define generic vanishing notions for $\Q$-twisted objects:

\begin{definicion}
We say that a $\Q$-twisted object $\mathcal{F}\left<t\ell\right>$ is 
\begin{enumerate}
\item\hspace{0.1cm} IT(0) if a (any) representant of $V^{i}(\mathcal{F}\left<t\ell\right>)$ is empty for $i\neq 0$
\item\hspace{0.1cm} GV if $\mathrm{codim}_{\hat{A}} V^{i}(\mathcal{F}\left<t\ell\right>)\geq i$ for every $i\geq 1$ 
\end{enumerate}
\end{definicion}

For example, if $\ell$ is a polarization, then from the index theorem (see \cite[III.16]{Mumford}, or Kodaira vanishing, in the caracteristic zero case) it follows that $\mathcal{O}_{A}\left<t\ell\right>$ is IT(0) for $t>0$. Another important example of these generic vanishing conditions regards the notion of separation of jets that we now recall:

\begin{definicion}
Let $X$ be a variety and $x\in X$ a closed point with ideal sheaf $I_{x}$. Given an integer $p\geq 0$ and a line bundle $L$ on $X$ we say that $L$ separates $p$-jets at $x$ if the restriction map 
$$H^{0}(X,L)\rightarrow H^{0}\left(X,L\otimes\mathcal{O}_{X}/I_{x}^{p+1}\right)$$
is surjective. 
\end{definicion}

In particular, $L$ separates $0$-jets at $x$ if and only if $x$ is not a basepoint of $L.$ Furthermore, if $x$ is not a basepoint of $L,$ then $L$ separates 1-jets at $x$ if and only if $L$ separates tangent vectors at $x$ (i.e the map between tangent spaces induced by the rational morphism $\phi_{L}:X --> \mathbb{P}H^{0}(L)$ is injective).

\begin{ejemplo}
\label{ejemplo-jets}
Let $A$ be an abelian variety, $x\in A$ be a closed point with ideal $I_{x}$ and $L$ an ample line bundle with class $\ell\in\NS(A).$ For a positive integer $p,$ taking cohomology to the exact sequence $0\rightarrow I_{x}^{p+1}\otimes L\rightarrow L\rightarrow\mathcal{O}_{A}/I_{x}^{p+1}\rightarrow 0$ we see that $H^{i}(I_{x}^{p+1}\otimes L\otimes P_{\alpha}) = 0$ for every $\alpha$ and every $i\geq 2$ and thus $V^{i}(I_{x}^{p+1}\left<\ell\right>)=\emptyset$ for $i\geq 2.$ On the other hand, we see that a representant for $V^{1}(I_{x}^{p+1}\left<\ell\right>)$ is given by the set
\begin{equation*}
\left\{\alpha\in\hat{A} : \hspace{0.2cm}\text{$L\otimes P_{\alpha}$ does not separate $p$-jets in $x$}\right\}.
\end{equation*}
Now, for an ample line bundle $L$ we have that $L\otimes P_{\alpha}\simeq t_{y}^{\ast}L$ for every $y\in\varphi_{L}^{-1}(\alpha).$ Using this, we see that, up to a traslation, the above set is equal to 
$$ \varphi_{L}\left(\left\{z\in A : L\hspace{0.1cm}\text{does not separate $p$-jets at $z$}\right\}\right).$$
It follows that $I_{x}^{p+1}\left<\ell\right>$ is GV if and only if there exists $z\in A$ such that $L$ separates $p$-jets at $z$ and it is IT(0) if and only if this happens for all $z\in A.$ In particular, $I_{x}\left<\ell\right>$ is always GV while it is IT(0) if and only if $L$ is basepoint free.
\end{ejemplo}

We have the following characterization of the behaviour of the generic vanishing conditions of $\mathcal{F}\left<t\ell\right>$ for varying $t$:

\begin{theorem}[\cite{cohrank}, Theorem 5.2]
\label{GVnoIT}
\begin{enumerate}
\item\hspace{0.1cm} A $\Q$-twisted object $\mathcal{F}\left<t_{0}\ell\right>$ is GV if and only if $\mathcal{F}\left<t\ell\right>$ is IT(0) for every $t>t_{0}.$
\item\hspace{0.1cm} If $\mathcal{F}\left<t_{0}\ell\right>$ is GV but not IT(0) then for every $t<t_{0}$ the $\Q$-twisted object $\mathcal{F}\left<t\ell\right>$ is not GV
\item\hspace{0.1cm} $\mathcal{F}\left<t_{0}\ell\right>$ is IT(0) if and only if $\mathcal{F}\left<(t_{0}-\eta)\ell\right>$ is IT(0) for every $\eta>0$ small enough
\end{enumerate}
\end{theorem}

There is also the following important result of preservation of vanishing:

\begin{theorem}[\cite{RegAVIII}, Theorem 3.2 and Remark 3.3 ]
\label{preservacion}
 Let $\mathcal{F},\mathcal{G}$ coherent sheaves on an abelian variety and let $\ell\in\NS(A)$ be a polarization. Write $\mathcal{F}\hspace{0.1cm}\underline{\otimes}\hspace{0.1cm}\mathcal{G}\in D^{b}(A)$ for the \emph{derived} tensor product of $\mathcal{F}$ and $\mathcal{G}.$ 

Then if both $\mathcal{F}\left<t\ell\right>$ and $\mathcal{G}\left<s\ell\right>$ are GV then $(\mathcal{F}\hspace{0.1cm}\underline{\otimes}\hspace{0.1cm}\mathcal{G})\left<(t+s)\ell\right>$ is  GV.
\end{theorem}

The conditions defined in this subsection can also be reformulated in terms of the Fourier-Mukai transform with kernel $\mathcal{P},$ which we briefly survey in the following subsection. 

\subsection{Fourier-Mukai transform}

If we write $p_{A},p_{\hat{A}}$ for the product projections from $A\times\hat{A},$ then we can consider the Fourier-Mukai functor $\Phi_{\mathcal{P}}: D^{b}(A)\rightarrow D^{b}(\hat{A})$ given by 
$$\mathcal{F}\mapsto Rp_{\hat{A}\ast}\left(p_{A}^{\ast}\mathcal{F}\otimes\mathcal{P}\right),$$
which by \cite{Mukai} is an equivalence of categories. For an object $\mathcal{F}\in D^{b}(A)$ we will write $R^{i}\Phi_{\mathcal{P}}(\mathcal{F})$ for the $i$-th cohomology sheaf of $\Phi_{\mathcal{P}}(\mathcal{F}).$ 

From the cohomology and base change theorem \cite[Théorème (7.7.5) II]{EGAIII-II} we have that
$$\mathrm{Supp}\hspace{0.1cm} R^{i}\Phi_{\mathcal{P}}(\lambda_{b}^{\ast}\mathcal{F}\otimes L^{ab})\subseteq\left\{\alpha\in\hat{A} : H^{i}(\lambda_{b}^{\ast}\mathcal{F}\otimes L^{ab}\otimes P_{\alpha})\neq 0\right\}.$$ 
In particular, we see that if $\mathcal{F}\left<\frac{a}{b}\ell\right>$ is IT(0) then $R^{i}\Phi_{\mathcal{P}}(\lambda_{b}^{\ast}\mathcal{F}\otimes L^{ab})=0$ for every $i\neq 0.$ Moreover, we have the following equivalence (\cite[Lemma 3.6]{PPo3}): 
$$\operatorname{codim}_{\hat{A}} \operatorname{Supp} R^{i}\Phi_{\mathcal{P}}(\lambda_{b}^{\ast}\mathcal{F}\otimes L^{ab})\geq i \iff \operatorname{codim}_{\hat{A}} V^{i}(\mathcal{F},L,a,b) \geq i.$$
On the other hand, if $c$ is a positive integer (not necessarily coprime to the characteristic of the base field), then by \cite[(3.4)]{Mukai} we have that 
$$\Phi_{\mathcal{P}}(\lambda_{bc}^{\ast}\mathcal{F}\otimes L^{abc^2})\simeq\Phi_{\mathcal{P}}(\lambda_{c}^{\ast}(\lambda_{b}^{\ast}\mathcal{F}\otimes L^{ab}))\simeq\lambda_{c\ast}\Phi_{\mathcal{P}}(\lambda_{b}^{\ast}\mathcal{F}\otimes L^{ab}).$$
In particular, as $\lambda_{c}$ is finite, we see that the following conditions are equivalent:
\begin{enumerate}[label = \alph*)]
\item \hspace{0.1cm} $\mathcal{F}\left<t\ell\right>$ is GV 
\item \hspace{0.1cm} $\operatorname{codim}_{\hat{A}} V^{i}(\mathcal{F},L,a,b)\geq i$ for every $i\geq 1$ and a (any) representation $t=a/b$ 
\item \hspace{0.1cm} $\operatorname{codim}_{\hat{A}} \operatorname{Supp} R^{i}\Phi_{\mathcal{P}}(\lambda_{b}^{\ast}\mathcal{F}\otimes L^{ab}) \geq i$ for every $i$ and a (any) representation $t=a/b$
\end{enumerate}

\begin{ejemplo}
\label{ejemplo-suc-jets}
Applying $\Phi_{\mathcal{P}}(-)$ to the exact sequence $0\rightarrow L\otimes I_{x}^{p}\rightarrow L\rightarrow L\otimes\mathcal{O}_{A}/I_{x}^{p}\rightarrow 0$ we get the exact sequence
$$0\rightarrow R^{0}\Phi_{\mathcal{P}}(I_{x}^{p}\otimes L)\rightarrow R^{0}\Phi_{\mathcal{P}}(L)\rightarrow R^{0}\Phi_{\mathcal{P}}(\mathcal{O}_{A}/I_{x}^{p})\rightarrow R^{1}\Phi_{\mathcal{P}}(I_{x}^{p}\otimes L)\rightarrow 0,$$
and the $R^1$ at the right is supported on the points where $L$ fails to separate $(p-1)$-jets.
\end{ejemplo}

\subsection{Cohomological rank functions}

For varying $t\in\Q,$ the fact that whether $V^{i}(\mathcal{F}\left<t\ell\right>) = \hat{A}$  or not,  can be expressed using the \emph{cohomological rank functions} introduced by Jiang-Pareschi in \cite{cohrank}. In this subsection we recall their definition and fix some notation
\begin{definicion}
\label{coh-rk}
Let $A$ be a $g$-dimensional abelian variety, $\ell\in\NS(A)$ an ample class and $\mathcal{F}\in D^{b}(A).$ Given $i\in\Z,$ the $i$-th cohomological rank function of $\mathcal{F}$ with respect to $\ell$ is the function $\Q\rightarrow\Q$ given by 
$$\frac{a}{b}\longrightarrow h^{i}\left(\mathcal{F}\left<\frac{a}{b}\ell\right>\right):= \frac{1}{b^{2g}}\cdot\mathrm{min}\left\{h^{i}(\lambda_{b}^{\ast}\mathcal{F}\otimes L^{ab}\otimes P_{\alpha}) : \alpha\in\hat{A}\right\},$$
where $L$ is a (any) line bundle representing $\ell.$
\end{definicion}
 
\begin{comentario}
\label{ceroGV}
The minimum involved in the definition is actually the generic value of $h^{i}(\lambda_{b}^{\ast}\mathcal{F}\otimes L^{ab}\otimes P_{\alpha})$ for varying $\alpha.$ In particular, from Example \ref{ejemplo-jets} we see that $I_{x}^{p}\left<t\ell\right>$ is GV if and only if $h^{1}(I_{x}^{p}\left<t\ell\right>) = 0.$ 
\end{comentario}

One of the main properties of the cohomological rank functions is the following fundamental transformation formula with respect to the Fourier-Mukai functor:

\begin{prop}[\cite{cohrank}, Proposition 2.3]
\label{transformation-formula}
For $\mathcal{F}$ and $\ell$ as in Definition \ref{coh-rk} above and $t>0,$ the following equalities hold:
\begin{enumerate}[label=\alph*)]
\item\hspace{0.1cm} $h^{i}(\mathcal{F}\left<-t\ell\right>) = \frac{t^{g}}{\chi(\ell)}\cdot h^{i}\left((\varphi_{\ell}^{\ast}\Phi_{\mathcal{P}}(\mathcal{F}))\left<\frac{1}{t}\ell\right>\right)$
\item\hspace{0.1cm} $h^{i}(\mathcal{F}\left<t\ell\right>) = \frac{t^{g}}{\chi(\ell)}\cdot h^{g-i}\left((\varphi_{\ell}^{\ast}\Phi_{\mathcal{P}^{\vee}}(\mathcal{F}^{\vee}))\left<\frac{1}{t}\ell\right>\right),$
where $\mathcal{F}^{\vee}$ stands for the \emph{derived} dual $R\shom(\mathcal{F},\mathcal{O}_{A}).$
\end{enumerate}
\end{prop}

Moreover, in Theorem 3.2 from \emph{loc-cit} it is also shown that cohomological rank functions can be extended to continous functions $\R\rightarrow\R.$

\section{Vanishing thresholds}

In this section we introduce the notion of vanishing thresholds for coherent sheaves. Afterwards, we give some examples and at the end we prove  Theorem \ref{A} 1) from the introduction. 

\begin{definicion}
\label{varios-thresholds}
\begin{enumerate}
\item\hspace{0.1cm} Let $A$ be an abelian variety and $\mathcal{F}\in D^{b}(A).$ Let $\ell$ be a polarization on $A.$ The vanishing threshold of $\mathcal{F}$ with respect to $\ell$ is the real number
$$\nu_{\ell}(\mathcal{F}):= \mathrm{inf}\left\{t\in\Q_{\geq 0} : \mathcal{F}\left<t\ell\right>\hspace{0.1cm}\text{satisfies IT(0)}\right\}.$$
\item\hspace{0.1cm} Given a closed subscheme $Z$ of $A$ and a non-negative integer $p$ we write 
$$\beta_{p}(Z,\ell) = \nu_{\ell}\left(I_{Z}^{p+1}\right),$$
where $I_{Z}$ is the ideal sheaf of $Z.$
\item\hspace{0.1cm} The $p$-jets separation threshold of $\ell$ is the number
$$\beta_{p}(\ell): = \beta_{p}(\{0\},\ell),$$
where we consider the reduced scheme structure on the set $\{0\}.$
\end{enumerate}
\end{definicion}
We note that, by definition, we have that $\nu_{n\ell}(\mathcal{F}) = n^{-1}\nu_{\ell}(\mathcal{F}).$ Also, thanks to Theorem \ref{GVnoIT} we have the following equality:
$$\nu_{\ell}(\mathcal{F}) = \mathrm{inf}\left\{t\in\Q_{\geq 0} : \mathcal{F}\left<t\ell\right>\hspace{0.1cm}\text{is GV}\right\}.$$

\begin{ejemplo}
\label{bpf}
We have that $\beta_{0}(\ell)$ is the basepoint-freeness threshold $\beta(\ell)$ considered in \cite[Section 8]{cohrank}. In particular $\beta_{0}(\ell)\leq 1$ and $\beta_{0}(\ell)=1$ if and only if $\ell$ has base points. More generally, from Theorem \ref{GVnoIT} we see that 
\begin{enumerate}
\item\hspace{0.1cm} $\beta_{p}(\ell)\leq 1$ if and only if there exists a representant $L$ of $\ell$ such that $L$ separates $p$-jets at 0. Arguing as in Example \ref{ejemplo-jets}, we see that this is equivalent to say that for any representant $L$ of $\ell$ there exists a point such that $L$ separates $p$-jets at such a point. 
\item\hspace{0.1cm} $\beta_{p}(\ell)<1$ if and only if for every representant $L$ of $\ell,$ $L$ separates $p$-jets at 0. Equivalently, any representant $L$ of $\ell$ separates $p$-jets at every point of $A.$
\end{enumerate} 
\end{ejemplo}

\begin{ejemplo}
\label{principal}
Consider a principal polarization $\theta$ and $\Theta$ a symmetric divisor representing it. If $\theta$ is indecomposable, then $\mathcal{O}_{A}(2\Theta)$ fails to separate 1-jets (i.e tangent vectors) just in the 2-torsion points of $A.$ That is, $I_{0}^{2}\left<2\theta\right>$ is GV but not IT(0) which, by Remark \ref{ceroGV} and Theorem \ref{GVnoIT}, means that $\beta_{1}(\theta) = 2.$ From \cite[p.99]{cav}, the same conclusion holds if $\theta$ is decomposable.
\end{ejemplo}

\begin{ejemplo}
Let $(A,\theta)$ be a $g$-dimensional principally polarized abelian variety and $\tau$ a length two closed subscheme supported in the origin. Again in this case, we have that 
$$\beta_{0}(\tau,\theta) = \inf\left\{t\in\Q_{\geq 0} : h^{1}(I_{\tau}\left<t\theta\right>) = 0\right\}.$$ 
Let $\Theta$ be a symmetric divisor representing $\theta$ and identify $A$ with $\hat{A}$ via $\varphi_{\theta}.$ We claim that $R^{j}\Phi_{\mathcal{P}^{\vee}}(I_{\tau}(\Theta)^{\vee}) = 0$ for $j\leq g-2,$ while $R^{g-1}\Phi_{\mathcal{P}^{\vee}}(I_{\tau}(\Theta)^{\vee}) = \mathcal{O}_{A}(-\Theta).$ To see this, notice that $I_{0/\tau}\simeq k(0)$ and hence we have an exact sequence 
$$0 \rightarrow I_{\tau}(\Theta)\rightarrow I_{0}(\Theta) \rightarrow k(0)\rightarrow 0.$$
As $\Phi_{\mathcal{P}^{\vee}}(I_{0}(\Theta)^{\vee})\simeq\mathcal{O}_{\Theta}(\Theta)[-g]$ and $\Phi_{\mathcal{P}^{\vee}}(k(0))\simeq\mathcal{O}_{A}[-g],$ we obtain that $R^{j}\Phi_{\mathcal{P}^{\vee}}(I_{\tau}(\Theta)^{\vee}) = 0$ for $j\leq g-2$ and an exact sequence
$$0\rightarrow R^{g-1}\Phi_{\mathcal{P}^{\vee}}(I_{\tau}(\Theta)^{\vee})\rightarrow\mathcal{O}_{A}\rightarrow\mathcal{O}_{\Theta}(\Theta)\rightarrow R^{g}\Phi_{\mathcal{P}^{\vee}}(I_{\tau}(\Theta)^{\vee})\rightarrow 0.$$
It follows that $R^{g-1}\Phi_{\mathcal{P}^{\vee}}(I_{\tau}(\Theta)^{\vee})\simeq\ker[\mathcal{O}_{A}\rightarrow\mathcal{O}_{\Theta}(\Theta)]\simeq\mathcal{O}_{A}(-\Theta).$ Now, for this example, the value of $R^{g}$ does not matter (however, see \cite[Corollary 4.3]{GV-min-GV}). Let $t\in(0,1]\cap\Q.$ By the transformation formula  (Proposition \ref{transformation-formula}b)) and the degeneration of the spectral sequence computing the hypercohomology we have that
$$h^{1}(I_{\tau}\left<(1+t)\theta\right>) = h^{1}(I_{\tau}(\Theta)\left<t\theta\right>) = t^{g}\cdot h^{g-1}\left(\Phi_{\mathcal{P}^{\vee}}(I_{\tau}(\Theta)^{\vee})\left<\frac{1}{t}\theta\right>\right) = t^{g}\cdot h^{0}\left(R^{g-1}\Phi_{\mathcal{P}^{\vee}}(I_{\tau}(\Theta)^{\vee})\left<\frac{1}{t}\theta\right>\right).$$ 
From our previous calculation we obtain then that
$$h^{1}(I_{\tau}\left<(1+t)\theta\right>)= t^{g}\cdot h^{0}\left(\mathcal{O}_{A}(-\Theta)\left<\frac{1}{t}\theta\right>\right) = t^{g}\cdot\left(\frac{1}{t}-1\right)^{g} = (1-t)^{g}.$$
As clearly $h^{1}(I_{\tau}\left<t\theta\right>)\neq 0$ for $t\in[0,1),$ we conclude that $$\beta_{0}(\tau,\theta) = 2.$$
\end{ejemplo}

\begin{ejemplo}
\label{curva}
Let $C$ be a smooth curve of genus $g$ and $u:C\hookrightarrow\operatorname{Jac}C$ an Abel-Jacobi map. By \cite[Theorem 7.5 and Proposition 7.6]{cohrank} we know that $h^{1}\left(\mathcal{O}_{u(C)}\left<t\theta\right>\right) = t^{g}-gt+(g-1)>0$ for $t\in[0,1).$ On the other hand, from the exact sequence
$$0\rightarrow\lambda_{b}^{\ast}I_{u(C)}\otimes\mathcal{O}_{\operatorname{Jac}(C)}(ab\Theta)\rightarrow\mathcal{O}_{\operatorname{Jac}(C)}(ab\Theta)\rightarrow\lambda_{b}^{\ast}\mathcal{O}_{u(C)}(ab\Theta)\rightarrow 0$$
it follows that 
$$h^{2}\left(I_{u(C)}\left<t\theta\right>\right) = h^{1}\left(\mathcal{O}_{u(C)}\left<t\theta\right>\right)>0\hspace{0.2cm}\text{for $t\in[0,1)$}.$$
This means that 
$$H^{2}\left(\lambda_{b}^{\ast}I_{u(C)}\otimes\mathcal{O}_{\operatorname{Jac}(C)}(ab\Theta)\otimes P_{\alpha}\right)\neq 0\hspace{0.3cm}\text{for every $\alpha\in\hat{A}$ and $a<b$}$$
and hence $I_{u(C)}\left<t\theta\right>$ is not GV for every $t<1.$ On the other hand, in \cite[Proposition 4.4]{RegAVI} it is shown that $I_{u(C)}\left<\theta\right>$ is GV and thus $\beta_{0}(u(C),\theta) = 1.$ 

Now, by \cite[Theorem 6.1]{GV-min-GV} this is esentially the only example of a curve inside an abelian variety whose GV-threshold is at most one. More precisely, if $C$ is a smooth curve that generates an abelian variety $A$ and $I_{C}\left<\theta\right>$ is GV for a principal polarization $\theta,$ then $C$ has minimal class and thus by Matsusaka's criterion \cite[Theorem 3]{Mats} it follows that $(A,\theta)\simeq(\operatorname{Jac}C,\theta_{C}).$ 

\end{ejemplo}

A couple of formal properties are the following:

\begin{prop}
\begin{enumerate}[label = \alph*)]
\item\hspace{0.1cm} If $Z$ is a smooth closed subscheme of an abelian variety $A$ then 
$$\beta_{p}(Z,\ell)\leq\mathrm{max}\left\{\beta_{p+1}(Z,\ell),\nu_{\ell}\left(\operatorname{Sym}^{p+1}N_{Z/A}^{\vee}\right)\right\},$$  
where $N_{Z/A}^{\vee}$ is the conormal sheaf of $Z$ in $A.$ 
\item\hspace{0.1cm} The sequence $\{\beta_{p}(\ell)\}_{p\in\Z_{\geq 0}}$ of jets-separation thresholds is unbounded and not decreasing
\end{enumerate}
\end{prop}

\begin{proof}

a) We want to prove that $I_{Z}^{p}\left<t\ell\right>$ is GV whenever both $I_{Z}^{p+1}\left<t\ell\right>$ and $(\operatorname{Sym}^{p}N_{Z/A}^{\vee})\left<t\ell\right>$ are GV.  As $Z$ is smooth, we have that 
$$I_{Z}^{p}/I_{Z}^{p+1}\simeq \operatorname{Sym}^{p}\left(I_{Z}/I_{Z}^{2}\right) = \operatorname{Sym}^{p}\hspace{0.1cm}N_{Z/A}^{\vee}$$
and thus we get exact sequences
$$0\rightarrow I_{Z}^{p+1}\rightarrow I_{Z}^{p}\rightarrow \operatorname{Sym}^{p}\hspace{0.1cm}N_{Z/A}^{\vee}\rightarrow 0$$
and thus the claim follows from \cite[Lemma 2.8 (1)]{itobir}.

b) The fact that the sequence is not decreasing follows from the previous part considering $Z = \{0\}$ (with reduced scheme structure). Now, the sequence $\{h^{0}(\mathcal{O}_{A}/I_{0}^{p+1})\}_{p\in\Z_{\geq 0}}$ is unbounded and for every $p\in\Z$ and for every line bundle $M$ we have 
$$M\otimes(\mathcal{O}_{A}/I_{0}^{p+1})\simeq\mathcal{O}_{A}/I_{0}^{p+1}.$$
It follows that, for any fixed $N\in\Z$ there exists $p\gg 0$ such that for every $\alpha\in\hat{A}$ the restriction map 
$$H^{0}(L^{N}\otimes P_{\alpha})\rightarrow H^{0}\left(L^{N}\otimes P_{\alpha}\otimes(\mathcal{O}_{A}/I_{0}^{p+1})\right)$$
is not surjective and thus $I_{0}^{p+1}\left<N\ell\right>$ can not be GV. In other words, for any $N$ there exists $p$ such that $\beta_{p}(\ell)\geq N$ and hence the sequence $\{\beta_{p}(\ell)\}_{p}$ is unbounded. 

\hfill\end{proof}

We also have the following:

\begin{theorem}
\label{elementales}
Let $\ell$ be a polarization on an abelian variety $A$ and $Z\subset A$ a closed subscheme of dimension at most one. Then for $p,r\in\Z$ positive integers with $p<r,$ the following inequalities hold:
$$0\leq\beta_{r}(Z,\ell)\leq\beta_{r-p-1}(Z,\ell)+\beta_{p}(Z,\ell).$$
In particular, the sequence $\{(p+1)^{-1}\beta_{p}(Z,\ell)\}_{p\in\Z_{\geq 1}}$ converges and 
$$\lim_{p\to\infty}\frac{\beta_{p}(Z,\ell)}{p+1} = \inf_{p}\frac{\beta_{p}(Z,\ell)}{p+1}.$$
\end{theorem}

In order to prove the theorem, the following subadditivity lemma is fundamental:

\begin{lemma}
\label{mi-subaditividad}
Let $I,J$ ideals sheaves of subschemes of dimension at most one. Let $M$ be a line bundle such that $I\hspace{0.1cm}\underline{\otimes}\hspace{0.1cm}J\otimes M$ is GV. Then both $I\otimes J\otimes M$ and $IJ\otimes M$ are GV. 
\end{lemma}

\begin{myproof}
For an object $\mathcal{F}$ we write $V^{i}(\mathcal{F}) = \{\alpha\in\hat{A} : H^{i}(\mathcal{F}\otimes P_{\alpha})\neq 0\}.$ In this context, the hypothesis means that 
$$\operatorname{codim}_{\hat{A}} V^{i}(I\hspace{0.1cm}\underline{\otimes}\hspace{0.1cm}J\otimes M) \geq i\hspace{0.2cm}\text{for all $i$}.$$
Now, we claim that 
\begin{equation}
\label{contenciones-deseadas}
V^{i}(IJ\otimes M)\subseteq V^{i}(I\otimes J\otimes M)\subseteq V^{i}(I\hspace{0.1cm}\underline{\otimes}\hspace{0.1cm}J\otimes M).
\end{equation}
It is clear that, once the claim is proved, the result follows. Now, to prove the claim, we prove the opposite contentions between the complements, that is, we show that if for $M_{\alpha}:=M\otimes P_{\alpha}$ we have $H^{i}(I\hspace{0.1cm}\underline{\otimes}\hspace{0.1cm}J\otimes M_{\alpha}) = 0$ then $H^{i}(I\otimes J\otimes M_{\alpha})=0$ and if such vanishing holds then $H^{i}(IJ\otimes M_{\alpha}) = 0.$ To do this we use the (fourth-quadrant) spectral sequence (\cite[(3.5)]{Huybrechts}) 
$$E_{2}^{p,q} = H^{p}(\shtor_{-q}(I,J)\otimes M_{\alpha})\implies H^{p+q}(I\hspace{0.1cm}\underline{\otimes}\hspace{0.1cm}J\otimes M_{\alpha}).$$
Tensoring the exact sequence $0\rightarrow I\rightarrow\mathcal{O}_{A}\rightarrow\mathcal{O}_{A}/I\rightarrow 0$ by $J$ we see that the support of the sheaves $\shtor_{i}(I,J)$ for $i>0$ is contained in the intersection of the cosupports of $I$ and $J.$ As such scheme has dimension at most one, it follows that
$$H^{p}(\shtor_{i}(I,J)\otimes M_{\alpha}) = 0\hspace{0.2cm}\text{for $i>0$ and $p\geq 2$}$$
and the spectral sequence already degenerates at the second page. It follows that $E_{2}^{i,0} = H^{i}(I\otimes J\otimes M_{\alpha})$ is a subquotient of $H^{i}(I\hspace{0.1cm}\underline{\otimes}\hspace{0.1cm}J\otimes M_{\alpha})$ and thus
$$V^{i}(I\otimes J\otimes M)\subseteq V^{i}(I\hspace{0.1cm}\underline{\otimes}\hspace{0.1cm}J\otimes M),$$
as we wanted to see. 

Now, to prove that $IJ\otimes M$ is also GV, we note that we have the exact sequence
$$0\rightarrow\shtor_{1}(I,J)\otimes M_{\alpha}\rightarrow I\otimes J\otimes M_{\alpha}\rightarrow IJ\otimes M_{\alpha}\rightarrow 0,$$
and thus taking cohomology we get an exact sequence
$$H^{i}(I\otimes J\otimes M_{\alpha})\rightarrow H^{i}(IJ\otimes M_{\alpha})\rightarrow H^{i+1}(\shtor_{1}(I,J)\otimes M_{\alpha}).$$
As $\dim\operatorname{Supp}\shtor_{1}(I,J)\leq 1$ the $H^{i+1}$ at the right is zero for $i\geq 1$ and hence we get the inclusion
$$V^{i}(IJ\otimes M)\subseteq V^{i}(I\otimes J\otimes M)$$
that we wanted to prove. 

\end{myproof}

\begin{myprooft}

Let $s,t\in\Q$ such that both $I_{Z}^{r-k}\left<t\ell\right>$ and $I_{Z}^{k+1}\left<s\ell\right>$ are GV. By \ref{preservacion} we have that $(I_{Z}^{r-k}\hspace{0.1cm}\underline{\otimes}\hspace{0.1cm}I_{Z}^{k+1})\left<(s+t)\ell\right>$ is GV. If we write $s=a/b$ and $t=c/e,$ the latter condition means that $\lambda_{be}^{\ast}(I_{Z}^{r-k}\hspace{0.1cm}\underline{\otimes}\hspace{0.1cm}I_{Z}^{k+1})\otimes L^{be(ae+bc)}$ is GV. Now, as pullbacks commutes with derived tensor products and $\lambda_{be}$ is a flat morphism (so $\lambda_{be}^{\ast}I_{Z}^{k+1}$ and $\lambda_{be}^{\ast}I_{Z}^{r-k}$ are also ideals), Lemma \ref{mi-subaditividad} above implies then that $I_{Z}^{r+1}\left<(s+t)\ell\right>$ is GV, as we wanted to see. 

Finally, from Fekete's subadditivity lemma (see, \cite[Lemma 1.4]{mustata}, for example), it follows that the sequence $\{(k+1)^{-1}\beta_{k}(Z,\ell)\}_{k}$ converges and 
$$\lim_{k\to\infty} \frac{\beta_{k}(Z,\ell)}{k+1} = \inf_{k}\frac{\beta_{k}(Z,\ell)}{k+1}.$$

\end{myprooft}

An immediate consequence is the following:

\begin{corollary}
\label{desigualdad-jets-obvia}
For every $p\in\Z_{\geq 0}$ the following inequalites of jets-separation thresholds hold:
$$\beta_{p}(\ell)\leq (p+1)\beta_{0}(\ell)\leq p+1$$
and if $\beta_{p}(\ell) = p+1$ then $\beta_{r}(\ell) = r+1$ for every $p\leq r.$ 
\end{corollary}

\begin{comentario}
\label{Bauer}
Note that using the above corollary we see that if $\beta_{0}(\ell)<(p+1)^{-1}$ then $\beta_{p}(\ell)<1$ and thus $\ell$ separates $p$-jets at every point. In particular we recover the well known fact (\cite[Theorem 1]{Bauer},\cite[Theorem (4)]{RegAVIII}) that for every $p\in\Z_{\geq 0}$ we have that $(p+2)\ell$ separates $p$-jets at every point. 
\end{comentario}

\section{Seshadri constants}

In Section 3 above we proved that for a closed subscheme $Z$ of dimension at most one, the sequence $\{(p+1)^{-1}\beta_{p}(Z,\ell)\}_{p}$ converges. In this section we establish that the limit is actually the inverse of the Seshadri constant of the ideal $I_{Z}$ with respect to $\ell$.  

First, recall that for an ample line bundle $L$ and an ideal sheaf $I,$ the Seshadri constant of $I$ with respect to $L$ is the real number
$$\varepsilon(I,L) = \sup\left\{t\in\Q : \sigma^{\ast}L - tE\hspace{0.1cm}\text{is nef}\right\},$$
where $\sigma:\mathrm{Bl}_{I}A\rightarrow A$ is the blow-up along $I$ with exceptional divisor $E.$

\begin{theorem}
\label{seshadrilimite}
Let $L$ be an ample line bundle on an a $k$-abelian variety $A$ and $\ell$ the corresponding polarization. Let $Z$ be a closed subscheme of $A.$ Then:    
$$\sup \frac{p+1}{\beta_{p}(Z,\ell)}\geq \varepsilon(I_{Z},L)\geq \lim\sup\frac{p+1}{\beta_{p}(Z,\ell)}.$$
If $\dim Z\leq 1$ then 
$$\varepsilon(I_{Z},L) = \lim_{p\to\infty}\frac{p+1}{\beta_{p}(Z,\ell)}=\sup_{p}\frac{p+1}{\beta_{p}(Z,\ell)}.$$
In particular, in this case we have that $\beta_{p}(Z,\ell) \geq (p+1)\varepsilon(I_{Z},L)^{-1}$ for every $p\in\Z_{\geq 0}.$
\end{theorem}

\begin{proof}

The proof of this statement closely follows \cite[Theorem 3.2]{ELC}. First we prove that $\varepsilon(I_{Z},L)\hspace{0.1cm}\leq\hspace{0.1cm}\sup_{p}(p+1)\beta_{p}(Z,\ell)^{-1}.$ In order to do this we consider a positive rational number $t = a/b <\varepsilon(I_{Z},L)$ with $\operatorname{char} k\nmid a$ (note that this kind of numbers form a dense subset of $(0,\varepsilon(I_{Z},L)$) and we need to show that there exists $k$ (possibly very big) such that $\beta_{k}(Z,\ell)\leq t^{-1}(k+1),$ for which is enough to find $k\gg 0$ such that 
\begin{equation}
\label{vanishingdeseado}
H^{j}(\lambda_{a}^{\ast}I_{Z}^{k+1}\otimes L^{ab(k+1)}) = 0 \hspace{0.2cm}\text{for $j>0$}.
\end{equation}
To do this, we start by noticing that by \cite[Lemma 3.3]{ELC} there exists $r_{0}$ such that
\begin{equation}
\label{lema-laz} 
H^{j}(I_{Z}^{r}\otimes M)\simeq H^{j}(\sigma^{\ast}M\otimes\mathcal{O}_{X^{\prime}}(-rE))
\end{equation}
for all $j\geq 0,$ all $r\geq r_{0}$ and all line bundles $M,$ where $X^{\prime}=\mathrm{Bl}_{Z}X$ and $\sigma:X^{\prime}\to X$ is the corresponding projection. Now, let $u$ be an integer such that $au\geq r_{0}.$ For all such $u$ consider $k_{u}= au-1$ and thus we have 
$$H^{j}(\lambda_{a}^{\ast}I_{Z}^{k_{u}+1}\otimes L^{ab(k_{u}+1)}) = H^{j}(\lambda_{a}^{\ast}I_{Z}^{au}\otimes L^{a^{2}bu})\simeq\bigoplus_{\alpha\in\hat{A}[a]} H^{j}(I_{Z}^{au}\otimes L^{bu}\otimes P_{\alpha}),$$
where in the last step we used the fact that $\operatorname{char} k\nmid a$ to get the direct sum (\cite[p.72]{Mumford}). So, in order to get the vanishing \eqref{vanishingdeseado}, we need to ensure the vanishing of each direct summand.
Now, by \eqref{lema-laz} and the way we choose $u$ we have that
\begin{align*}
H^{j}(I_{Z}^{au}\otimes L^{bu}\otimes P_{\alpha}) & \simeq H^{j}(\sigma^{\ast}(L^{bu}\otimes P_{\alpha})\otimes\mathcal{O}_{X^{\prime}}(-auE)) \\
& \simeq H^{j}((\sigma^{\ast}P_{\alpha})\otimes\mathcal{O}_{X^\prime}(b\sigma^{\ast}L-aE)^{\otimes u}).     
\end{align*}
Now, as $\varepsilon(I_{Z},L)>a/b$ we have that $b\sigma^{\ast}L-aE$ is ample and thus, as we just need to consider finite $\alpha$'s, by Serre's vanishing we can take $u\gg 0$ such that the desired vanishings \eqref{vanishingdeseado} hold. Summarizing, we have established that $I_{Z}^{k_{u}+1}\left<\frac{b(k_{u}+1)}{a}\ell\right>$ is IT(0) for $u\gg 0$ and thus 
$$\sup_{k} \frac{k+1}{\beta_{k}(Z,\ell)} \geq \frac{k_{u}+1}{\beta_{k_u}(Z,\ell)}\geq\frac{a}{b} = t.$$
As $t$ can be arbitrarily close to $\varepsilon(I_{Z},L)$ we conclude that 
$$\sup_{k} \frac{k+1}{\beta_{k}(Z,\ell)}\geq\varepsilon(I_{Z},L).$$
To prove the opposite inequality, we first note that by \cite[Proposition 5.4.5]{PositivityI} we have that 
$$(k+1)\varepsilon(I_{Z},L)^{-1} \leq d_{L}(I_{Z}^{k+1}) := \min\{d\in\Z_{\geq 0} : I_{Z}^{k+1}\otimes L^{\otimes d}\hspace{0.1cm}\text{is globally generated}\},$$
while Mumford's theorem (\cite[Theorem 1.8.5]{PositivityI}) says that 
$$d_{L}(I_{Z}^{k+1})\leq\mathrm{reg}_{L}(I_{Z}^{k+1}):=\min\{m\in\Z_{\geq 0} : I_{Z}^{k+1}\hspace{0.1cm}\text{is $m$-regular with respect to $L$}\},$$
where being $m$-regular with respect to $L$ means that $H^{j}(I_{Z}^{k+1}\otimes (m-j)L) = 0$ for $j>0$. Now, by definition of $\beta_{k}(Z,L)$ and Theorem \ref{GVnoIT}, we have that $I_{Z}^{k+1}\left<t\ell\right>$ is IT(0) for every $t>\beta_{k}(Z,L)$. In particular, for any integer $r>\beta_{k}(Z,L)$ we have the following vanishings:  
$$H^{j}(I_{Z}^{k+1}\otimes L^{\otimes r}) = 0\hspace{0.1cm}\text{for all $j>0$}.$$
For $j> g$ we automatically have $H^{j}(I_{Z}^{k+1}\otimes L^{\otimes r}) = 0$  and hence $I_{Z}^{k+1}$ is $m$-regular with respect to $L$ as soon as 
$$ m - j > \beta_{k}(Z,L) \hspace{0.4cm}\text{for every $j\in\{1,...,g\}$},$$
for which it is enough to have $m > g+\beta_{k}(Z,L).$ In particular, we obtain that  
$$\mathrm{reg}_{L}(I_{Z}^{k+1})\leq 1+g + \lceil\beta_{k}(Z,L)\rceil\leq 2+g+\beta_{k}(Z,L).$$
It follows that 
$$\varepsilon(I_{Z},L)^{-1}\leq\frac{2+g}{k+1}+\frac{\beta_{k}(Z,L)}{k+1}\hspace{0.3cm}\text{for all $k$}.$$
Passing to the limit we get
$$\varepsilon(I_{Z},L)^{-1}\leq\lim\inf\frac{\beta_{k}(Z,L)}{k+1} = \frac{1}{\lim\sup\frac{k+1}{\beta_{k}(Z,L)}}$$ 
and therefore 
$$\varepsilon(I_{Z},L) \geq \lim\sup\frac{k+1}{\beta_{k}(Z,L)},$$
as we wanted to see.

If $\dim Z\leq 1,$ Theorem \ref{elementales} above implies that we have equalities.
 
\hfill\end{proof}

\section{Gauss-Wahl maps}

In this section we recall the definition of the Gauss-Wahl maps and in particular we introduce some thresholds that may be thought as the surjectivity thresholds of such maps. Then we establish Theorem \ref{B} from the introduction, which relates these thresholds with the jets-separation thresholds.

To start, consider a smooth projective variety $X.$ We write $p_{1},p_{2}:X\times X\rightarrow X$ for the projections, $\Delta\subset X\times X$ for the diagonal and $I_{\Delta}\subset\mathcal{O}_{X\times X}$ for the corresponding ideal sheaf. For $n\in\Z_{\geq 0}$ we write $n\Delta$ for the closed subscheme of $X\times X$ defined by the ideal $I_{\Delta}^{n},$ (we consider $I_{\Delta}^{0} = \mathcal{O}_{X\times X}$) thus
$$\mathcal{O}_{n\Delta} = \mathcal{O}_{X\times X}/I_{\Delta}^{n}.$$
\begin{definicion}
Let $L$ be a line bundle on $X.$ Given $n\in\Z_{\geq -1},$ the sheaf of $n$-principal parts of $L$ over $X$ (or the sheaf of $n$-jets of $L$ over $X$) is the sheaf
$$P^{n}(L) := p_{1\ast}\left(p_{2}^{\ast}L\otimes\mathcal{O}_{(n+1)\Delta}\right).$$
We also consider the sheaf of $n$-jets-relations
$$R_{L}^{(n)} = p_{1\ast}\left(p_{2}^{\ast}L\otimes I_{\Delta}^{n+1}\right).$$
\end{definicion}

From the definition it follows that $P^{-1}(L) = 0$ and $P^{0}(L) = L.$ Now, applying $p_{1\ast}(p_{2}^{\ast}L\otimes -)$ to the exact sequence $0\rightarrow I_{\Delta}^{n+1}\rightarrow \mathcal{O}_{X\times X}\rightarrow\mathcal{O}_{(n+1)\Delta}\rightarrow 0$ we get an exact sequence
\begin{equation}
\label{defRL}
\xymatrix{ 0\ar[r] &  R_{L}^{(n)}\ar[r] &  H^{0}(L)\otimes\mathcal{O}_{X}\ar[r]^-{\mathrm{ev}} & P^{n}(L),}
\end{equation}
where the fiber over $x$ of the last arrow is given by the restriction
$$H^{0}(L)\rightarrow H^{0}\left(L\otimes\mathcal{O}_{X}/I_{x}^{n+1}\right).$$
In particular, the right arrow is surjective if and only if $L$ separates $n$-jets at every point of $X.$  In this way we also see that $R_{L}^{(-1)} = H^{0}(L)\otimes\mathcal{O}_{X}$ and $R_{L}^{(0)}$ is the kernel of the evaluation map (which in \cite{cohrank} is denoted by $M_{L}$).
On the other hand, applying $p_{1\ast}(p_{2}^{\ast}L\otimes -)$ to the exact sequence
$$0\rightarrow I_{\Delta}^{n+1}\rightarrow I_{\Delta}^{n}\rightarrow I_{\Delta}^{n}/I_{\Delta}^{n+1}\rightarrow 0$$
we get an exact sequence
\begin{equation}
\label{defgauss}
\xymatrix{0\ar[r] & R_{L}^{(n)}\ar[r] & R_{L}^{(n-1)}\ar[r]^-{u} & L\otimes\operatorname{Sym}^{n}\Omega_{X},}
\end{equation}
fitting in the following commutative and exact diagram:
\begin{equation}
\label{diagramote}
\xymatrix{ & 0\ar[d] & 0\ar[d] & & \\
                & R_{L}^{(n)}\ar[d]\ar@{=}[r] & R_{L}^{(n)}\ar[d] & & \\
 0\ar[r] & R_{L}^{(n-1)} \ar[r]\ar[d]_{u} & H^{0}(L)\otimes\mathcal{O}_{X}\ar[r]^{\mathrm{ev}_{n-1}}\ar[d]_{\mathrm{ev}_n} & P^{n-1}(L)\ar@{=}[d] &  \\ 
0\ar[r] & L\otimes\operatorname{Sym}^{n}\Omega_{X}\ar[r] & P^{n}(L)\ar[r] & P^{n-1}(L)\ar[r] & 0 }             
\end{equation}
where the surjectivity at the bottom follows from the fact that the sheaf $I_{\Delta}^{n}/I_{\Delta}^{n+1}$ is supported on the diagonal (for example, see \cite[Proposition 1.3]{Ricolfi} for details). In particular, from the 5-lemma it follows that $u$ is surjective as soon as the evaluation map $\mathrm{ev}_{n}$ is surjective, that is, as soon as $L$ separates $n$-jets at every point of $X.$

\begin{definicion}
Given $L,M$ line bundles on a smooth projective variety $X$ and $n\in\Z_{\geq 0}.$
\begin{enumerate}[label = \alph*)]
\item\hspace{0.1cm} The $n$-th Gauss-Wahl map associated to $L,M$ is the map 
$$\gamma_{L,M}^{n} = H^{0}(u\otimes M) : H^{0}(R_{L}^{(n-1)}\otimes M)\rightarrow H^{0}(L\otimes M\otimes\operatorname{Sym}^{n}\Omega_{X}),$$
where $u$ is the morphism in \eqref{defgauss}.
\item\hspace{0.1cm} The $n$-th multiplication map associated to $L,M$ is the map
$$m_{L,M}^{n} = H^{0}(\mathrm{ev}_{n}\otimes M): H^{0}(L)\otimes H^{0}(M)\rightarrow H^{0}(P^{n}(L)\otimes M).$$
\end{enumerate}
\end{definicion}

From the diagram \eqref{diagramote} we see that the surjectivity of $m_{L,M}^{n}$ implies the surjectivity of $\gamma_{L,M}^{n}.$ Moreover, from \eqref{defgauss} and the remark below it, we see that if $L$ generate $n$-jets at every point then the surjectivity of $m_{L,M}^{n}$ is equivalent to the vanishing of $H^{1}(R_{L}^{(n)}\otimes M)$ and hence such vanishing implies the surjectivity of $\gamma_{L,M}^{n}.$ In the case that $M = L^{k}$ and $X$ is an abelian variety, this suggests to study the cohomological rank function $h^{1}(R_{L}^{(n)}\left<-\hspace{0.1cm}\ell\right>),$ where $\ell$ is the class of $L$ in $\NS(X).$ 

\begin{lemma}
\label{GV-RP}
Let $L$ be an ample line bundle on an abelian variety of dimension $g$. Then: 
\begin{enumerate}[label = \alph*)]
\item\hspace{0.1cm} For $t>-1$ we have that $P^{n}(L)\left<t\ell\right>$ is IT(0) and $h^{0}(P^{n}(L)\left<t\ell\right>) = (1+t)^{g}h^{0}(L){n+g\choose g}$
\item\hspace{0.1cm} If $L$ separates $n$-jets at every point then we have $h^{j}(R_{L}^{(n)}\left<s\ell\right>) = 0$ for every $s>0$ and $j\geq 2$ while
$$h^{1}(R_{L}^{(n)}\left<s\ell\right>) \geq (1+s)^{g}h^{0}(L){n+g\choose g}$$ 
for $s\in(-1,0),$ with equality if $g\geq 2.$
\end{enumerate}
\end{lemma}

\begin{proof}
a)  Write $t=a/b$ with $a>-b$ and $b>0.$ Without loss of generality we may assume that $L$ is symmetric and hence $\lambda_{b}^{\ast}L\simeq L^{b^2}.$ Applying $\lambda_{b}^{\ast}( - )\otimes L^{ab}\otimes P_{\alpha}$ to the exact sequences
\begin{equation}
\label{suc-pes}
0\rightarrow L\otimes\operatorname{Sym}^{n}\Omega_{A}\rightarrow P^{n}(L)\rightarrow P^{n-1}(L)\rightarrow 0
\end{equation}
we get that for every $j\geq 1$ we have
$$H^{j}(\lambda_{b}^{\ast}P^{n}(L)\otimes L^{ab}\otimes P_{\alpha})\simeq H^{j}(\lambda_{b}^{\ast}P^{n-1}(L)\otimes L^{ab}\otimes P_{\alpha})$$
and thus, inductively, we see that 
$$H^{j}(\lambda_{b}^{\ast}P^{n}(L)\otimes L^{ab}\otimes P_{\alpha})\simeq H^{j}(\lambda_{b}^{\ast}P^{0}(L)\otimes L^{ab}\otimes P_{\alpha}) = H^{j}(L^{b(a+b)}\otimes P_{\alpha}) = 0,$$
as we wanted to prove. It also follows, inductively, that 
\begin{equation}
\label{ecuacion}
h^{0}(\lambda_{b}^{\ast}P^{n}(L)\otimes L^{ab}\otimes P_{\alpha})  = h^{0}(L^{b(a+b)}\otimes P_{\alpha})\sum_{k=0}^{n}{g+k-1\choose g-1} = b^{g}(b+a)^{g}h^{0}(L){n+g\choose g}.
\end{equation}

b) Now, write $s = c/d$ with $d$ a positive integer and $c>-d$. As $L$ separates $n$-jets at every point, the sequence in \eqref{defRL} is also right exact. As $H^{i}(L^{cd}\otimes P_{\alpha}) = 0$ for $0<i<g,$ applying $\lambda_{d}^{\ast}(-)\otimes L^{cd}\otimes P_{\alpha}$ and taking cohomology we obtain then the following isomorphisms
$$H^{j}(\lambda_{d}^{\ast}R_{L}^{(n)}\otimes L^{cd}\otimes P_{\alpha})\simeq H^{j-1}(\lambda_{d}^{\ast}P^{n}(L)\otimes L^{cd}\otimes P_{\alpha}) =0 \hspace{0.2cm}\text{for every $2\leq j\leq g-1$},$$
where the last equality comes from part a). Moreover, if $c>0$ then $H^{g}(L^{cd}\otimes P_{\alpha}) = 0$ and hence the above isomorphism also holds for $j=g,$ as soon as $g\geq 2.$ Now, if $c<0$ then $H^{0}(L^{cd}\otimes P_{\alpha}) = 0$ and therefore we have the following exact sequence:
\begin{equation}
\label{casonegativo}
0\rightarrow H^{0}(\lambda_{d}^{\ast}P^{n}(L)\otimes L^{cd}\otimes P_{\alpha})\rightarrow H^{1}(\lambda_{d}^{\ast}R_{L}^{(n)}\otimes L^{cd}\otimes P_{\alpha})\rightarrow H^{0}(L)\otimes H^{1}(L^{cd}\otimes P_{\alpha}),
\end{equation} 
where the group at the right is zero if $g\geq 2.$ The result then follows from part a).

\hfill\end{proof}

In the spirit of \cite[Section 8]{cohrank} we introduce the following numbers:

\begin{definicion}
\label{def-threshold}
Let $L$ be an ample line bundle on an abelian variety $A.$ For $p\in\Z_{\geq 0}$ we write
$$\mu_{p}(L) = \inf\left\{t\in(-1,\infty)\cap\Q : h^{1}(R_{L}^{(p)}\left<t\ell\right>) =0\right\},$$
where $\ell$ is the class of $L$ in $\NS(A).$
\end{definicion}

Note that, when $L$ separates $p$-jets at every point, the above lemma tells us that $\mu_{p}(L)$ is the vanishing threshold  $\nu_{\ell}(R_{L}^{(p)})$ (see Definition \ref{varios-thresholds}) and, in this context, $\mu_{p}(L)\geq 0.$ Note also that the invariant $\mu_{0}(L)$ is already considered in \cite{cohrank}, where it is denoted by $s(L).$ In the cited reference it is shown that there is a relation between such number and the base-point freeness threshold, when $L$ is globally generated. In the following we compute $\Phi_{\mathcal{P}}(I_{0}^{p+1}\otimes L)$ to show that there is an analagous relation between $\mu_{p}$ and $\beta_{p}.$

Before stating our result, we need to introduce further notation.
Let $L$ be an ample line bundle on an abelian variety $A$ and $\alpha\in\hat{A}.$ For $a,b\in\Z_{>0}$ and $n\in\Z_{\geq 0}$ we write 
$$m_{L,\alpha}^{n}(a,b) = H^{0}(\lambda_{b}^{\ast}\mathrm{ev}_{n}\otimes L^{\otimes ab}\otimes P_{\alpha}): H^{0}(L)\otimes H^{0}(L^{ab}\otimes P_{\alpha})\rightarrow H^{0}(\lambda_{b}^{\ast}P^{n}(L)\otimes L^{\otimes ab}\otimes P_{\alpha}),$$
and for $t = a/b\in\Q$ we write
$$\operatorname{gcorank} m_{L}^{n}(t) = \frac{1}{b^{2g}}\cdot \min_{\alpha\in\hat{A}}\operatorname{corank} m_{L,\alpha}^{n}(a,b),$$
and it is easy to see that this number does not depend on the particular representation $t=a/b.$ As observed before, when $L$ separates $p$-jets at every point, $\operatorname{gcorank} m_{L}^{p}(t)$ is no other than $h^{1}(R_{L}^{(p)}\left<t\ell\right>).$    

\begin{theorem}
\label{relacion}
Let $A$ be an abelian variety. Let $L$ be an ample line bundle on $A$ with class $\ell\in\NS(A).$ Then for every $y\in(0,1)$ the following equality holds:
$$h^{1}\left(I_{0}^{p+1}\left<y\ell\right>\right) = \frac{(1-y)^{g}}{h^{0}(L)}\cdot\operatorname{gcorank} m_{L}^{n}\left(\frac{y}{1-y}\right).$$
In particular, if $L$ separates $p$-jets at every point, the following equality holds:
$$\mu_{p}(L) = \frac{\beta_{p}(\ell)}{1-\beta_{p}(\ell)}.$$
and if $L$ and $L^{\prime}$ are algebraically equivalent then $\mu_{p}(L) = \mu_{p}(L^{\prime}).$
\end{theorem}

The main tool to prove this is the following lemma:

\begin{lemma}
\label{transformada-fundamental}
Let $A$ be an abelian variety and $L$ an ample line bundle over $A.$ We write $0^n$ for the closed subscheme of $A$ defined by the ideal $I_{0}^{n}.$ Then we have 
$$\varphi_{L}^{\ast}\Phi_{\mathcal{P}}\left[\xymatrix{L\ar[r]^-{\mathrm{res}} & L\otimes \mathcal{O}_{0^n}}\right]\simeq \left[\xymatrix{ H^{0}(L)\otimes\mathcal{O}_{A}\ar[r]^-{\mathrm{ev}} & P^{n-1}(L)}\right]\otimes L^{\vee}.$$
\end{lemma}

\begin{proof}
Let $m:A\times A\rightarrow A$ be the multiplication map and $p_{1},p_{2}:A\times A\rightarrow A$ be the projections. 

By \cite[(3.10)]{Mukai} we have that 
$$\varphi_{L}^{\ast}\Phi_{\mathcal{P}}\left[L\otimes\left(\xymatrix{\mathcal{O}_{A}\ar[r]^{res} & \mathcal{O}_{0^n}}\right)\right]\otimes L \simeq m_{\ast}\left[p_{1}^{\ast}\left(\xymatrix{\mathcal{O}_{A}\ar[r]^{res} & \mathcal{O}_{0^n}}\right)\otimes p_{2}^{\ast}L\right].$$
Now, let $\mu:A\times A\rightarrow A\times A$ be the map given by $\mu(x,y) = (m(x,y),y),$ whose inverse is given by $\nu(x,y) = (\delta(x,y),y),$ where $\delta:A\times A\rightarrow A$ is the difference map. Then we have:
\begin{align*}
m_{\ast}\left[p_{1}^{\ast}\left(\xymatrix{\mathcal{O}_{A}\ar[r]^{res} & \mathcal{O}_{0^n}}\right)\otimes p_{2}^{\ast}L\right] & \simeq p_{1\ast}\mu_{\ast}\left[p_{1}^{\ast}\left(\xymatrix{\mathcal{O}_{A}\ar[r]^{res} & \mathcal{O}_{0^n}}\right)\otimes p_{2}^{\ast}L\right] \\
& \simeq p_{1\ast}\left[(p_{1}\circ\nu)^{\ast}\left(\xymatrix{\mathcal{O}_{A}\ar[r]^{res} & \mathcal{O}_{0^n}}\right)\otimes (p_{2}\circ\nu)^{\ast}L\right] \\
& \simeq p_{1\ast}\left[\delta^{\ast}\left(\xymatrix{\mathcal{O}_{A}\ar[r]^{res} & \mathcal{O}_{0^n}}\right)\otimes p_{2}^{\ast}L\right] \\
& \simeq \xymatrix{H^{0}(L)\otimes\mathcal{O}_{A}\ar[r]^-{ev} &  P^{n-1}(L)},
\end{align*}
where in the last line we used the fact that $\delta^{\ast}\mathcal{O}_{0^n} = \mathcal{O}_{n\Delta}$ and the definition of the natural map in \eqref{defRL}.
 
\hfill\end{proof}

\begin{myprooft}

We have that $I_{0}^{p+1}\otimes L$ is isomorphic in $D^{b}(A)$ to the complex $L\rightarrow L\otimes\mathcal{O}_{0^{p+1}}$ (concentrated in degrees 0 and 1). By Lemma \ref{transformada-fundamental} above, this implies that in $D^{b}(A)$ we have the following isomorphism
$$\varphi_{L}^{\ast}\Phi_{\mathcal{P}}(I_{0}^{p+1}\otimes L) \simeq \mathrm{EV}_{p}^{\bullet}\otimes L^{\vee},$$
where $\mathrm{EV}_{p}^{\bullet}$ is the complex (concentrated in degrees 0 and 1) given by
$$\xymatrix{H^{0}(L)\otimes\mathcal{O}_{A} \ar[r]^{\mathrm{ev}} & P^{p}(L)}.$$
By the transformation formula (Proposition \ref{transformation-formula}) it follows that for $s\in\Q^{-}$ we have
$$h^{1}\left((I_{0}^{p+1}\otimes L)\left<s\ell\right>\right) = \frac{(-s)^{g}}{h^{0}(L)}\cdot h^{1}\left((\mathrm{EV}_{p}^{\bullet}\otimes L^{\vee})\left<-\frac{1}{s}\ell\right>\right).$$
Now, by Lemma \ref{GV-RP} above, we have that both $(P^{p}(L)\otimes L^{\vee})\left<-\frac{1}{s}\ell\right>$ and $L^{\vee}\left<-\frac{1}{s}\ell\right>$ are IT(0) for $s\in(-1,0).$ In this context, writing $s=-a/b$ with $b>a>0,$ the spectral sequence (\cite[Remark 2.67]{Huybrechts})
$$E_{1}^{r,q} = H^{q}(\lambda_{a}^{\ast}(\mathrm{EV}_{p}^{r}\otimes L^{\vee})\otimes L^{ab}\otimes P_{\alpha})\implies H^{r+q}(\lambda_{a}^{\ast}(\mathrm{EV}_{p}^{\bullet}\otimes L^{\vee})\otimes L^{ab}\otimes P_{\alpha})$$ 
says that 
$$H^{1}(\lambda_{a}^{\ast}(\mathrm{EV}_{p}^{\bullet}\otimes L^{\vee})\otimes L^{ab}\otimes P_{\alpha})\simeq\operatorname{Coker} m_{L,\alpha}^{p}(b-a,a)$$
and thus 
$$h^{1}\left((\mathrm{EV}_{p}^{\bullet}\otimes L^{\vee})\left<-\frac{1}{s}\ell\right>\right) = \operatorname{gcorank} m_{L}^{p}\left(-\left(1+\frac{1}{s}\right)\right).$$
Finally, as $h^{1}((I_{0}^{p+1}\otimes L)\left<s\ell\right>) = h^{1}(I_{0}^{p+1}\left<(1+s)\ell\right>),$ setting $y = 1+s,$ the result follows.

\end{myprooft}

An important consequence of Theorem \ref{relacion} is that we can ensure the surjectivity of certain Gauss-Wahl maps when we know that $\beta_{p}$ is small. We discuss this application in detail in the next section. For the moment, we limit ourselves to discuss formal consequences of the above theorem and its relation with the Seshadri constant.

First, we note that, unlike the jets-separation thresholds, a priori it is not clear how $\mu_{p}(L)$ changes when we replace $L$ by a multiple of it. However, using the previous theorem, we are able to give such a formula:

\begin{corollary}
\label{multiplos}
Let $L$ be an ample line bundle and suppose that $L$ separates $p$-jets at every point. Then for every $n\in\Z_{\geq 1}$ we have that
$$\mu_{p}(nL) = \frac{\mu_{p}(L)}{n+(n-1)\mu_{p}(L)}.$$ 
\end{corollary} 

\begin{proof}
As $L$ separates $p$-jets, it follows that also $nL$ separates $p$-jets for every positive $n$ and hence, applying the previous theorem we get that the following equalities hold:
\begin{align*}
\mu_{p}(nL) & = \frac{\beta_{p}(n\ell)}{1-\beta_{p}(n\ell)}= \frac{n^{-1}\beta_{p}(\ell)}{1-n^{-1}\beta_{p}(\ell)}  \\
& = \frac{\beta_{p}(\ell)}{n-\beta_{p}(\ell)} = \frac{\frac{\mu_{p}(L)}{1+\mu_{p}(L)}}{n-\frac{\mu_{p}(L)}{1+\mu_{p}(L)}} \\
& = \frac{\mu_{p}(L)}{n+(n-1)\mu_{p}(L)}.
\end{align*}

\hfill\end{proof}

Combining with the results from previous section we get an expression of the Seshadri constant in terms of the thresholds $\mu_{p}.$

\begin{corollary}
\label{seshadri-gauss}
Let $L$ be an ample line bundle on an abelian variety, then 
$$\varepsilon(L) = \varepsilon(I_{x},L) = 1+\lim_{p\to\infty}\frac{1}{\mu_{p}((p+2)L)},$$
where $\mu_{p}$ is the threshold given in Definition \ref{def-threshold}.
\end{corollary}

\begin{proof}
From Remark \ref{Bauer} we see that if $L$ is an ample line bundle then $L^{\otimes(p+2)}$ separates $p$-jets at every point. Therefore, by Theorem \ref{relacion} we have
$$\mu_{p}\left((p+2)L\right) = \frac{\beta_{p}\left((p+2)\ell\right)}{1-\beta_{p}\left((p+2)\ell\right)}.$$
On the other hand, by definition we know that $\beta_{p}((p+2)\ell) = (p+2)^{-1}\beta_{p}(\ell).$ Therefore we have
$$\mu_{p}((p+2)L) =  \frac{\beta_{p}(\ell)}{2+p-\beta_{p}(\ell)} = \frac{\frac{\beta_{p}(\ell)}{p}}{\frac{2+p}{p}-\frac{\beta_{p}(\ell)}{p}}.$$
In particular the limit $\lim_{p\to\infty}\mu_{p}((p+2)L)$ exists as an element of $\R\cup\{\infty\}$ and by Proposition \ref{seshadrilimite} it follows that
$$\lim_{p\to\infty}\mu_{p}((p+2)L) = \frac{\frac{1}{\varepsilon(\ell)}}{1-\frac{1}{\varepsilon(\ell)}} =  \frac{1}{\varepsilon(\ell)-1}$$
or, in other words, we have 
$$\varepsilon(\ell) = 1 + \lim_{p\to\infty} \frac{1}{\mu_{p}((p+2)L)},$$
as we wanted to see.

\hfill\end{proof}

Using the results of this section, we get the following:

\begin{corollary}
\label{Nakamaye}
Let $A$ be an abelian variety of dimension $g$ and a polarization $\ell\in\NS(A).$ Then the following are equivalent:
\begin{enumerate}[label =\alph*)]
\item\hspace{0.1cm} $\varepsilon(\ell) = 1$
\item\hspace{0.1cm} $\beta_{p}(\ell) = p+1$ for every $p\in\Z_{\geq 0}$
\item\hspace{0.1cm} For $L$ a line bundle representing $\ell$ the sequence $\{\mu_{p}((p+2)L)\}_{p\in\Z_{\geq 0}}$ is unbounded 
\item\hspace{0.1cm} There exists a principally polarized elliptic curve $(E,\theta)$ and a $g-1$-dimensional polarized abelian variety $(B,\underline{m})$ such that 
$$(A,\ell)\simeq(E,\theta)\boxtimes(B,\underline{m})$$
\end{enumerate}
\end{corollary}

\begin{proof}

 The equivalence between a) and b) follows as a combination of Theorem \ref{elementales} and Theorem \ref{seshadrilimite}.  From Corollary \ref{seshadri-gauss} above, the equivalence a)$\iff$c) is immediate. Finally, the equivalence a)$\iff$d) is the content of Nakamaye's theorem \cite[Theorem 1.1]{Nakamaye}. 

\hfill\end{proof}

\section{Effective surjectivity of Gauss-Wahl maps}

In this section we establish Theorem \ref{resultado-sobreyectividad} from the introduction.

\begin{theorem}
\label{sobreyectividad-jets}
Let $L$ and $M$ be ample and algebraically equivalent line bundles on an abelian variety $A.$ Let $c,d\in\Z_{>0}$ and $p\in\Z_{\geq 0}.$ Then the Gauss-Wahl map
\begin{equation}
\label{Gauss-considerado}
\gamma_{cL,dM}^{p}:H^{0}(R_{cL}^{(p-1)}\otimes M^{\otimes d})\rightarrow H^{0}(L^{\otimes c}\otimes M^{\otimes d}\otimes\operatorname{Sym}^{p}\Omega_{A})
\end{equation}
is surjective whenever 
$$\beta_{p}(\ell) < \frac{dc}{c+d}.$$  
In particular, if $L$ separates $p$-jets at every point then $\gamma_{2L,2M}^{p}$ is surjective. 
\end{theorem}

\begin{proof}
The surjectivity of \eqref{Gauss-considerado} follows as soon as $\mu_{p}(cL)<d/c.$ Indeed, by Lemma \ref{GV-RP}b) this inequality means that the $\Q$-twisted sheaf $R_{cL}^{(p)}\left<d\ell\right>$ is IT(0), where $\ell$ is the class of $L$ and $M$ in $\NS(A),$ which in particular means that $H^{1}\left(R_{cL}^{(p)}\otimes M^{\otimes d}\right) = 0$ and thus that $\gamma_{cL,dM}^{p}$ is surjective. 

Now, we note that by hypothesis we have 
$$\beta_{p}(\ell) < \frac{dc}{c+d} < c,$$
which means that $\beta_{p}(c\ell)<1,$ that is, $cL$ separates $p$-jets at every point of $A.$  From Theorem \eqref{relacion} it follows that 
$$\mu_{p}(cL) = \frac{\beta_{p}(cL)}{1-\beta_{p}(cL)} = \frac{\beta_{p}(L)}{c-\beta_{p}(L)}.$$
Now, as the function $[0,c)\rightarrow\R : t\mapsto \frac{t}{c-t}$ is strictly increasing, we have
$$\mu_{p}(cL)< \frac{\frac{dc}{c+d}}{c-\frac{dc}{c+d}} = \frac{d}{c},$$
and hence the result follows. 

Regarding the last part, we note that hypothesis means that $\beta_{p}(\ell)<1.$ On the other hand, if $c,d\geq 2$ then $cd/(c+d)\geq 1$ and thus the result follows. 

\hfill\end{proof}

Finally, we use multiplier ideals in order to bound the jets-separation thresholds $\beta_{p}(\ell)$ in terms of the Seshadri constant of $\ell.$  To do this, we will need the following:
 
\begin{lemma}
Let $A$ be an abelian variety and $L$ an ample line bundle on it. Fix a rational number $t_{0} = u_{0}/v_{0}$ with $t_{0}<\varepsilon(L).$ Then for every $c\in\Q$ with $0<c\ll 1$ and for every $m\gg 0$ divisible enough integer, there exists a divisor $D_{0}$ such that 
\begin{enumerate}[label = \roman*)]
\item\hspace{0.1cm} $D_{0}\in\left|mv_{0}(1-c)L\right|$
\item\hspace{0.1cm} $\mathrm{mult}_{0}D_{0} = mu_{0}$
\item\hspace{0.1cm} $\mu:\mathrm{Bl}_{0}A\rightarrow A$ is a log-resolution of $D_{0}$ and the strict transform $\mu_{\ast}^{-1}D_{0}$ is reduced
\end{enumerate}  
\end{lemma}

\begin{proof}
Up to minor modifications, this is the content of \cite[Lemma 1.2]{LPP}.
\hfill\end{proof}

Now, if we fix $c\ll 1$ and $m\gg 0$ as in the previous lemma then considering $\Delta = \frac{g+p}{mu_{0}}D_{0}$ and $r\in\Q$ with $r\geq\frac{g+p}{\varepsilon(L)}\left(1-\frac{c}{2}\right),$ we obtain: 
\begin{align*}
\mu^{\ast}\Delta & = \frac{g+p}{mu_{0}}\mu_{\ast}^{-1}D_{0}+ \left(\frac{g+p}{mu_0}\cdot\mathrm{mult}_{0}D_{0}\right)E \\
& = \frac{g+p}{mu_{0}}\mu_{\ast}^{-1}D_{0} + (g+p)E.
\end{align*}
Now, as $\mu$ is a log-resolution of $D_{0}$ we have that $\mu_{\ast}^{-1}D_{0}+E$ is snc and hence, as $m\gg 0,$ it follows that 
$$\lfloor\mu^{\ast}\Delta\rfloor = (g+p)E.$$
In this way we can compute the multiplier ideal 
$$\mathcal{J}(A,\Delta) = \mu_{\ast}\mathcal{O}_{\mathrm{Bl}_0 X}\left((g-1)E-(g+p)E\right) = \mu_{\ast}\mathcal{O}_{\mathrm{Bl}_0 X}(-(p+1)E) = I_{0}^{p+1}.$$
On the other hand:
$$rL-\Delta \equiv \left(r-\frac{g+p}{t_0}(1-c)\right)L,$$
which is ample because of the way in we chose $r.$ Therefore, the following inequality directly follows from Nadel's vanishing:

\begin{prop}
\label{desigualdad}
For a polarization $\ell$ on an abelian variety the following inequality holds
$$\beta_{p}(\ell)\hspace{0.1cm}\leq r_{p}(L)\leq\hspace{0.1cm}\frac{g+p}{\varepsilon(L)},$$
where 
$$r_{p}(L):=\mathrm{inf}\left\{r\in\Q\left| \begin{matrix}  \text{$\exists$ an effective $\Q$-divisor $\Delta$ such that $\mathcal{J}(X,\Delta)=I_{0}^{p+1}$} \\ \text{and $rL-\Delta$ is ample} \end{matrix}\right.\right\},$$
and $L$ is a line bundle representing $\ell.$
\end{prop}

Now, from Proposition \ref{desigualdad} and Theorem \ref{sobreyectividad-jets} we are able to prove the following:

\begin{corollary}
\label{sobreyectividad-efectiva}
Let $L,M$ be ample and algebraically equivalent line bundles and $p\in\Z_{\geq 0}.$ Consider a positive integer $c$ such that $\varepsilon(cL)>g+p.$ Then the Gauss-Wahl map $\gamma_{cL,dM}^{p}$ is surjective as soon as
\begin{equation}
\label{dmayor} 
d\hspace{0.1cm}>\hspace{0.1cm}\frac{c(g+p)}{c\varepsilon(L)-(g+p)}.
\end{equation}
\end{corollary}

\begin{proof}

From Theorem \ref{sobreyectividad-jets} and Proposition \ref{desigualdad}, the desired surjectivity follows as soon as 
$$\frac{g+p}{\varepsilon(\ell)} < \frac{cd}{d+c},$$
where $\ell$ is the class of $L$ and $M.$ Now, as $\varepsilon(c\ell)>g+p$ the above inequality becomes \eqref{dmayor}.

\hfill\end{proof}

To illustrate the usage of this result, let $L,M$ be algebraically equivalent ample line bundles as in the theorem. Consider $f,h:\Z_{\geq 0}\rightarrow\Z_{\geq 0}$ two functions and suppose that we want to ensure the surjectivity of $\gamma_{cL,dM}^{p}$ for all $c\geq f(p)$ and $d\geq h(p).$ In order to apply the corollary we would just need 
$$\varepsilon(\ell)\hspace{0.1cm}>\hspace{0.1cm}\frac{g+p}{f(p)}\hspace{0.2cm}\text{and}\hspace{0.2cm}h(p)\hspace{0.1cm} >\hspace{0.1cm} \frac{(g+p)f(p)}{f(p)\varepsilon(\ell)-(g+p)}$$
or, equivalently:
$$\varepsilon(\ell)\hspace{0.1cm}>\hspace{0.1cm}\frac{(f(p)+h(p))(g+p)}{f(p)h(p)}.$$
As an example, we contrast with \cite[Theorem 2.2]{mult}. In such reference it is proved that $\gamma_{cL,dM}^{p}$ is surjective for all $c,d\geq 2(p+1)$ with $c+d\geq 4p+5.$ Now, if we consider $f(p) = h(p) = 2p+1$ we get that $\gamma_{cL,dM}^{p}$ is surjective for all $c,d\geq 2(p+1)$ (no matter $c+d$) whenever
\begin{equation}
\label{desigualdad-paraprop}
\varepsilon(\ell)\hspace{0.1cm}>\hspace{0.1cm} \frac{g+p}{p+1/2}
\end{equation}
As $(g+p)/(p+1/2)\to 1$ as $p\to\infty,$ from Nakamaye's theorem we get that if $(A,\ell)$ is not as in \eqref{tipo-producto} then \eqref{desigualdad-paraprop} holds for $p\gg 0.$ In this sense, Theorem \ref{sobreyectividad-jets} above asymptotically improves \cite[Theorem 2.2]{mult}.

\section{Further questions}

\textbf{1.} In Corollary \ref{Nakamaye} we saw that a special kind of polarized abelian varieties, namely the products 
\begin{equation}
\label{tipo-especial}
(A,\ell)\simeq (E,\theta)\boxtimes(B,m),
\end{equation}
 are characterized by the equality $\beta_{p}(\ell) = p+1$ for \emph{all} $p.$ Now, by \cite[Proposition 2]{Steffens} we know that if $(A,\theta)$ is a principally polarized abelian surface (p.p.a.s) then $\varepsilon(\theta) = 4/3$ whenever $(A,\theta)$ is not a product. From Proposition \ref{desigualdad} this means that for indecomposable p.p.a.s and $p\geq 2$ we have 
$$\beta_{p}(\theta)\hspace{0.1cm}<\hspace{0.1cm}\frac{2+p}{4/3}\hspace{0.1cm}\leq\hspace{0.1cm} p+1,$$
and thus, a posteriori, the polarized surfaces of the form \eqref{tipo-especial} are characterized by the equality $\beta_{2}(\theta) = 3.$ 
\begin{pregunta}
\label{unouno}
Let $g$ be a positive integer. Does there exist a positive integer $p_{0} = p_{0}(g)$ such that a $g$-dimensional principally polarized abelian variety $(A,\theta)$ is as in \eqref{tipo-especial} if and only if $\beta_{p_0}(\theta) = p_{0}+1$?
\end{pregunta}
More generally, we may ask:
\begin{pregunta}
Given a positive real number $u.$ Does there exist a positive real number $t(u)$ and a positive integer $p(u)$ such that for any $g$-dimensional principally polarized abelian variety $(A,\theta)$ we have
$$\varepsilon(\theta) > u \iff \beta_{p(u)}(\theta) < t(u) \hspace{0.3cm}\text{?}$$
(Note that Question \ref{unouno} corresponds to the case $u=1$)  
\end{pregunta}

\textbf{2.} Note that, except for Proposition \ref{desigualdad}, whose proof uses Nadel's vanishing theorem, the results presented in this article hold over an algebraically closed field of any characteristic. We may then ask the following:
\begin{pregunta}
\label{nadel-positiva}
Let $(A,\ell)$ be a polarized abelian variety over a field of positive characteristic. Does the inequality $\beta_{p}(\ell)\leq (g+p)\varepsilon(\ell)^{-1}$ still hold?
\end{pregunta}
On the other hand, in \cite{Mustata-Schwede} and in \cite{Murayama}, there are introduced the \emph{Frobenius-Seshadri constants} $\varepsilon_{F}^{k}(L,x)$ of a line bundle $L$ at a point $x$ of a smooth variety $X$ over an algebraically closed field of positive characteristic and it is shown that the following inequalities hold:
$$\frac{k+1}{\varepsilon(L,x)}\hspace{0.1cm}\leq\hspace{0.1cm}\frac{k+1}{\varepsilon_{F}^{k}(L,x)}\hspace{0.1cm}\leq\hspace{0.1cm}\frac{k+\dim X}{\varepsilon(L,x)},$$
where $\varepsilon(L,x)$ is the usual Seshadri constant at $x$. This, together with our Theorem \ref{A} 3), suggests to study the following:
\begin{pregunta}
Let $L$ be an ample line bundle on an abelian variety defined over an algebraically closed field of positive characteristic. Compare the number $(k+1)/\varepsilon_{F}^{k}(L,x)$ and the $k$-jets separation threshold $\beta_{k}(\ell),$ where $\ell$ is the class of $L.$
\end{pregunta}
For instance, if we are able to prove that 
\begin{equation}
\label{desigualdadsensata}
\beta_{k}(\ell)\hspace{0.1cm}\leq\hspace{0.1cm}\frac{k+1}{\varepsilon_{F}^{k}(L,x)},
\end{equation}
then Question \ref{nadel-positiva} would have positive answer.

\section{Acknowledgements} 

This work is part of the author's Phd thesis. The author would like to thank his advisor, Giuseppe Pareschi, for his invaluable guidance and help to improve both the exposition and the mathematical content of this article. The author also thanks to the anonymous referees for their comments, which were very useful to improve the presentation of this article and to clarify some proofs. This work was supported by the MIUR Excellence Department Project MatMod@TOV awarded to the Department of Mathematics of the University of Rome Tor Vergata.

\bibliographystyle{alpha}
\bibliography{referencias-seshadrigauss}

\end{document}